\documentclass[pdflatex]{amsart}
\usepackage{tikz}

\usetikzlibrary{arrows,shapes}

\usetikzlibrary{cd}

\usepackage{amsthm,amssymb,verbatim,amsmath, amsfonts, amscd,color}
\usepackage{amscd}
\usepackage{caption}

\usepackage{color}

\usepackage{verbatim}
\theoremstyle{plain}\newtheorem{Theorem}{Theorem}[section]
\theoremstyle{plain}
\theoremstyle{plain}\newtheorem{Corollary}[Theorem]{Corollary}
\theoremstyle{plain}\newtheorem{Lemma}[Theorem]{Lemma}
\theoremstyle{plain}
\theoremstyle{definition}\newtheorem{Definition}[Theorem]{Definition}
\theoremstyle{definition}
\theoremstyle{definition}
\theoremstyle{definition}\newtheorem{Remark}[Theorem]{Remark}
\theoremstyle{definition}
\theoremstyle{definition}\newtheorem{Notation}[Theorem]{Notation}
\theoremstyle{definition}
\theoremstyle{definition}
\theoremstyle{definition}
\theoremstyle{definition}
\theoremstyle{definition}
\theoremstyle{definition}
\theoremstyle{definition}\newtheorem{Notation/Definition}
[Theorem]{Notation/Definition}
\theoremstyle{definition}

\def\F{{\mathcal{F}}}

\def\Aut{\mathrm{Aut}}
 \def\Syl{\mathrm{Syl}}

\def\dim{\mathrm{dim}}           
\def\End{\mathrm{End}}           
\def\Ext{\mathrm{Ext}}           

\def\Hom{\mathrm{Hom}}

\def\Im{\mathrm{Im}}

\def\Irr{\mathrm{Irr}}

\def\mod{\mathrm{mod}}
\def\rad{\mathrm{rad}}

\def\soc{\mathrm{soc}}

\def\proj{\mathrm{(proj)}}
\def\H{\mathrm{H}}  
\def\D{\mathrm{D}}
\def\C{\mathrm{C}}
\def\K{\mathrm{K}}
\def\mod-{\mathrm{mod}\text{-}}

\newcommand{\SL}{\operatorname{SL}}

\newcommand{\GL}{\operatorname{GL}}

\newcommand{\Sc}{{\text{Sc}}}

\newcommand{\Out}{{\text{Out}}}
\newcommand{\Ker}{{\text{Ker}}}

\begin{document}

\title{the Mathieu group ${\sf M}_{12}$ vs ${\mathbf{SL_3(3)}}$ in characteristic $3$}
%The principal $3$-blocks of the Mathieu group ${\sf M}_{12}$ and the special linear group
%${\SL_3(3)}$ are derived equivalent}
%\date{\today}
\author{Shigeo Koshitani and Tetsuro Okuyama}
\address{S. Koshitani, Department of Mathematics and Informatics,
Graduate School of Science, Chiba
University, 1-33 Yayoi-cho, Inage-ku,Chiba 263-8522, Japan.}
\email{koshitan@math.s.chiba-u.ac.jp} 
\address{T. Okuyama, who passed away on 18th August 2024.}
\email{}

\thanks
{The first author
was partially supported by the Japan Society for Promotion of Science (JSPS),
Grant-in-Aid for Scientific Research (C)19K03416, 2019--2023.
\\
%$^*$The second author passed away on 18th August 2024.
}

%\keywords
%{
%}

%\dedicatory{}

\begin{abstract} 
We prove that the principal $3$-blocks of the Mathieu group ${\sf M_{12}}$ and 
the special linear group $\SL_3(3)$ are splendidly Rickard equivalent, 
and hence derived equivalent,
by applying Rickard's theorem whose origin goes back to the second author.
As a byproduct we answer to a question on the first Hochschild cohomology
of the principal $3$-blocks above posed by W. Murphy.
\end{abstract}

\maketitle

\pagestyle{myheadings} 
\markboth{S. Koshitani and T. Okuyama}
{The principal 3-blocks of $M_{12}$ and $\SL_3(3)$}

\section{Introduction} 
In modular representation theory of finite groups one of the most important and interesting
conjectures is (so-called) Brou\'e's Abelian Defect Group Conjecture (ADGC). Roughly speaking
it says that, for a prime number $p$, 
if the two principal $p$-blocks $A$ and $B$ of finite groups $G$ and $H$,
respectively, have the same {\sf abelian} defect group $P$ (hence $P$ is a common 
Sylow $p$-subgroup
of $G$ and $H$),
and if furthermore the fusion systems $\mathcal F_P(G)$ and $\mathcal F_P(H)$ over 
$P$ are the same,
then the algebras $A$ and $B$ should be derived equivalent. 
So one might ask what about if we look at a case that $P$ is 
{\sf non-abelian}, such as $p_+^{1+2}$ that is the extra-special group
of order $p^3$ with exponent $p$. Many people actually have looked at this case, for instance
\cite{Kos86, Usa01, RV04, Hen07, NU09, AE25}.
As a matter of fact, around four decades ago the first author classified 
all the possible (not necessary principal but general)
$p$-blocks $A$ with defect group $P=p_+^{1+2}$ for any odd prime $p$, up to
Morita equivalence, whenever $G$ is $p$-solvable \cite[Propositions 1.1 and 1.4]{Kos86}. 
Going back to our subject, Y.~Usami already in 1997 announced that
she proved that there is a perfect isometry between the principal $3$-blocks of 
the Mathieu group ${\sf M_{12}}$
and the special linear group $\SL_3(3)$ (see \cite[Proposition 4]{Usa01} 
and also \cite[Proposition 59]{NU09}). 
The existence of prefect isometries is
a weaker version of Brou\'e's ADGC, by the way.

Actually, in this paper we prove that a stronger conclusion than Usami's holds, 
namely, our main result is:

\begin{Theorem}\label{MainTheorem}
The principal $3$-blocks of the Mathieu group ${\sf M}_{12}$ and 
the special linear group $\SL_3(3)$ are
splendidly Rickard (splendidly derived) equivalent.
\end{Theorem}
\noindent
Though our result is on very specific two finite simple groups, we emphasize that the origin of
the theorem in \cite{Ric02} we heavily depend on goes back to the second author
as Rickard says many times.
In fact, the first author did not realize that such very concrete calculations
would be useful/powerful to check an unknown fact 
after over a few decades.

As a corollary the derived equivalence lifts to the bigger groups that have $3'$-index $2$;

\begin{Corollary}\label{Cor}
The principal $3$-blocks of $\Aut({\sf M}_{12})\cong {{\sf M}_{12}\rtimes 2}$ and
$\Aut(\SL_3(3))\cong {\SL_3(3)\rtimes 2}$ are splendidly Rickard (splendidly derived)
equivalent.
\end{Corollary}
\noindent
Theorem~\ref{MainTheorem} and a well-known fact that
a derived equivalence preserves the Hochschild cohomology structures (due to 
Happel) % and Rickard) 
immediately imply the
following corollary W.~Murphy asks in \cite{Mur23}.

\begin{Corollary}\label{HH^1}{\rm{(Cf}}.\,\cite[Remark\ 5.3]{Mur23}{\rm{)}}
The first Hochschild cohomology groups
\linebreak
${\mathrm{HH}}^1(B_0(k{\sf M}_{12}))$ and ${\mathrm{HH}}^1(B_0(k\SL_3(3)))$,
respectively,
of the principal $3$-blocks 
of ${\sf M}_{12}$ and 
of $\SL_3(3)$
are isomorphic as Lie algebras, where $k$ is an algebraically closed field of characteristic $3$.
\end{Corollary}

%%\pagebreak
\section{Notation and the structure of this paper}
We use the following notation and conventions.
Throughout this paper, $k$ is an algebraically closed field of a prime characteristic $p$. 
We write $A$ and $B$ for finite-dimensional $k$-algebras.
We denote by $1_A$ and $A^\times$, respectively, 
the unit element of $A$ and the set of all units in $A$.
By an $A$-{\sf module}, we mean a finitely generated 
{\sf right} $A$-module unless we specify otherwise.
Sometimes we use the notation $X_A$, $_A\!Y$ and $_AU_B$ 
to emphasize that $X$ is a {\sf right}
$A$-module, $Y$ is a {\sf left} $A$-module and $U$ is an $(A,B)$-{\sf bimodule}.
We write $\mod-A$ for the abelian category of finitely generated $A$-modules, and
${\mathrm{proj}}\text{-}A$ for the category of finitely generated projective $A$-modules. 
We denote by ${\mathrm{rad}}(A)$ and ${\mathrm{soc}}(A)$ the Jacobson radical and the socle
of $A$, respectively. 
Fix $X\in\mod-A$ with $X\,{\not=}\,0$.
Then, ${\mathrm{soc}}(X)$ denotes the socle of $X$.
We write $L_i(X)$ and $S^i(X)$, respectively, 
for the $i$-th Loewy layer and socle layer of $X$ for
each $i=1,2,\cdots$. More precisely,
$L_i(X):=X{\cdot}J^i/X{\cdot}J^{i-1}$ where $J:={\mathrm{rad}}(A)$,
$S^1(X):={\mathrm{soc}}(X)$ and $S^i(X)/S^{i-1}(X):={\mathrm{soc}}(X/S^{i-1}(X))$
for $i=2,3,\cdots$. For such an $X$ we denote by $\ell\ell(X)$ the Loewy length of $X$,
namely $\ell\ell(X)$ is the smallest positive integer $\ell$ with $X{\cdot}\rad(A)^\ell=0$.
We write $P(X)$ and $I(X)$, respectively, for 
the projective cover and the injective hull (envelope) of $X$.
We denote by $\Irr_k(A)$ the set of all non-isomorphic simple (left) $A$-modules.
For $S_1,\cdots,S_n\in\Irr_k(A)$ where $n$ is a positive number 
(possibly $S_i\cong S_j$ for some $i\,{\not=}\,j$)
we write
$ M=U(S_1, S_2, \cdots,S_n)$
for a uniserial $A$-module $M$
such that $L_i(M)\cong S_i$ for $i=1,\cdots,n$
(note that it may happen two $A$-modules $X$ and $Y$ have the same descriptions 
as above though $X$ and $Y$ are non-isomorphic).
Similarly, 
for $X\in\mod-A$ with $X\,{\not=}\,0$, we write $X=m_1\times S_1+\cdots+m_\ell\times S_\ell$
(as composition factors) when all the composition factors of $X$ are just as this 
where $\ell$ and $m_1, \cdots ,m_\ell$ are positive integers and $S_i\in\Irr_k(A)$ such that
$S_1,\cdots, S_\ell$ are all non-isomorphic each other.

For $S, T\in\Irr_k(A)$, $c_{S,T}$ denotes the Cartan invariant with respect to $S$ and $T$,
in other words, $c_{S,T}:=\dim_k\,\Hom_A(P(S),P(T))$. More generally for the same $X$
as in the previous paragraph
and $S$,
we write $c_X(S)$ for the multiplicity of the composition factors of $X$ that are isomorphic to $S$.
Let $\mathcal C$ be an (abelian or triangulated) category.
For such $\mathcal C$ we simply write $X^\bullet\in\mathcal C$ if $X^\bullet$ is an object
of $\mathcal C$. We denote by {\sf stmod}$A$ the stable module category of $A$-modules, 
which is the quotient
category of $\mod-A$ by the ideal of maps that factor through projective modules.
This is a triangulated category whose shift functor is the inverse $\Omega^{-1}$ of the Heller
translate (operator) $\Omega$. 
In particular, for $X\in\mod-A$, $\Omega^0X$ denotes the projective-free part of $X$,
that is $X=\Omega^0X\oplus(\mathrm{proj})$ and $\Omega^0X$ has no 
projective indecomposable
direct summand.
Fix $X, Y\in\mod-A$. Then we denote by 
$\underline{\Hom}_A(X,Y)$ the space of the
homomorphisms from $X$ to $Y$ in {\sf stmod}$A$,
namely,
$\underline{\Hom}_A(X,Y):=\Hom_A(X,Y)/\mathcal P\Hom_A(X,Y)$,
where $\mathcal P\Hom_A(X,Y)$ is the set of all $f\in\Hom_A(X,Y)$ such that $f$ is 
relatively projective. 
We write $X=Y\oplus{\mathrm{\proj}}$ when $X$ is a direct sum of
$Y$ and a certain projective $A$-module.
We write $Y\,|\,X$ if $Y$ is a direct summand of $X$.
Complexes treated here are {\sf cochain} complexes, that means that the differentials have
degree $+1$.
For an abelian category $\mathcal A$, we denote by 
$\C^b(\mathcal A)$, $\K^b(\mathcal A)$ and $\D^b(\mathcal A)$, respectively,
the category of bounded complexes, the bounded homotopy category 
and the bounded derived category of $\mathcal A$.
For $M,N\in\mod-A$, $\delta_{M,N}$ is (so-called) the Kronecker delta.

For a positive integer $n$ we write $C_n, \mathfrak S_n$ and $\mathfrak A_n$ 
for the cyclic group of order $n$, 
the symmetric group of degree $n$ and the alternating group of degree $n$, respectively.
Further, we denote by $D_{2^n}$ for an integer $n\geq 3$ and $SD_{2^n}$ for an integer $n\geq 4$,
respectively, the dihedral group and the semidihedral group of order $2^n$.
Let $G$ be a finite group and assume that $H\leq G$.
We write $\Delta H$ for the diagonal subgroup $\{(h,h)\in H\times H\,|\,h\in H\}$ of $H\times H$.
As standard convention, $O_{p'}(G)$ denotes the largest normal $p'$-subgroup of $G$, and further
$\Aut(G)$ and $\Out(G)$, respectively, denote the automorphism group and the outer-automorphism
group of $G$.
We denote by $C_G(H)$ and $N_G(H)$, respectively, the centralizer and the normalizer of $H$ in $G$,
in particular $Z(G):=C_G(G)$ is the center of $G$. 
For $H, L\leq G$, by $G=H\rtimes L$ we mean a semi-direct product of $H$ and $L$ 
where $H\unlhd G$.
We write $\Syl_p(G)$ for the set of all Sylow $p$-subgroups of $G$.
%For $H\leq G$, we denote by $C_G(H)$ and $ N_G(H)$ the centralizer and the normalizer,
%respectively, of $H$ in $G$, and $Z(G)$ denotes the center of $G$.
For $g, x\in G$ we define $x^g:=g^{-1}xg$, and 
for a subset $\mathcal S\subseteq G$ we define $\mathcal S^g:=\{s^g\,|\,s\in\mathcal S\}$.
For $g, g'\in G$ we write $g=_G g'$ when $g'=g^x$ for some $x\in G$, and for a subset 
$S\subseteq G$ and $g\in G$ we write $g\in_G S$ if $g=_G s$ for some $s\in S$. 
For a $p$-subgroup $P$ of $G$, $\mathcal F_P(G)$ denotes the fusion system over $P$
with respect to $G$-conjugation, more precisely, for subgroups $Q, R\leq P$,
$\Hom_{\mathcal F_P(G)}(Q,R):=\{ c_g\,|\, g\in G\text{ such that }Q^g\subseteq R\}$,
where $c_g(u):=g^{-1}ug$ for $u\in Q$ (see \cite[8.1]{Lin18}).

Even in this paragraph we keep the assumption $G\geq H$ as above.
We write $k_G=k$ for the one-dimensional trivial $kG$-module
and $\Irr(G)$ for the set of all irreducible ordinary characters of $G$.
For $M\in\mod-kG$ and $N\in\mod-kH$, we write $M{\downarrow}_H$ and 
$N{\uparrow}^G$ for the restriction of $M$
to $H$ and the induction of $N$ to $G$, respectively. 
We sometimes write also like $M{\downarrow}^G_H$ and $N{\uparrow}_H^G$ to emphasize the
original groups.
For the same $M$, the $k$-dual of $M$
(which is isomorphic to the $kG$-dual of $M$ since $kG$ is a symmetric $k$-algebra)
is denoted by $M^*$ that becomes also a {\sf right} $kG$-module.
We write $\Omega_H$ for the relatively $H$-projective Heller functor (operator), and hence,
for $M\in\mod-kG$, $\Omega_H(M)$ is the kernel of the epimorphism
$P_H(M)\twoheadrightarrow M$ where $P_H(M)$ is the relatively $H$-projective cover of $M$
(see \cite[2.7]{Lin18}).
In particular, we define the heart $\mathcal H_H(k_G)$ of the trivial $kG$-module $k_G$ by
$\mathcal H_H(k_G):=\Omega_H(k_G)/k_G$.
We denote by $B_0(kG)$ the principal block (algebra) of the group algebra $kG$, and
we write $\Irr (G)$ and $\Irr(B_0(kG))$, respectively, for the sets of all irreducible ordinary characters
of $G$ and of those which belong to $B_0(kG)$.
For a $p$-subgroup $P\leq G$ and $H\leq G$ with $H\geq N_G(P)$ 
we denote by $f_{(G,P,H)}$ the Green correspondence with respect to the triple
$(G,P,H)$ and we define
$\mathfrak X:=\mathfrak X(G,P,H)$ and also $\mathfrak Y$ and $\mathfrak Z$ just as in
\cite[Chap.4 (4.1) and Theorem 4.3]{NT88}. 
Finally, for $M, M'\in\mod-kG$, we denote $\dim_k\,\Hom_{kG}(M,M')$ by $[M,M']^G$.
For the other convention and terminology, see the books
\cite{NT88, 
Lin18}.

The structure of this paper is the following.
In \S3 we list several general lemmas which shall be useful for out main results.
In \S4 we introduce almost of all the notation we use throughout the whole paper,
and also we list many lemmas which play important roles to prove our main theorem.
In \S\S 5-9 we construct complexes in the module categories of the principal
$3$-blocks of the two specific finite simple groups ${\sf M}_{12}$ and $\SL_3(3)$ 
%we are looking at 
by making full use of the results (\cite{Kos86, Kos87, KW99a, KW99b})
done by the first author (partly with Waki) over several decades ago.
Finally in \S 10 proofs of our main results are presented. 

\section{The general recipe}

\begin{Notation}\label{Notation9}
Throughout this section we assume that $k$ is an algebraically closed field
of characteristic $p>0$ and that $A$ is a finite-dimensional self-injective $k$-algebra.
\end{Notation}

\noindent
Actually the next lemma is quite useful to apply Rickard's theorem \cite[Theorem 6.1]{Ric02}
when our complexes $X^\bullet$ have homology groups $\H^i(X^\bullet)$ such that
$\H^i(X^\bullet)=0$ except $i=0, 1$.
\begin{Lemma}\label{Zimmer-6.6.2}{\rm (}See \cite[Proposition~6.6.2]{Zim14}{\rm .)} \ 
%Let $A$ be a finite-dimensional $k$-algebra, and 
Let $X^\bullet$ and $Y^\bullet$ be objects in $\D^{\mathrm{b}}({\mathrm{mod}}\text{-}A)$,
(and hence objects in $\C^{\mathrm{b}}({\mathrm{mod}}\text{-}A)$).
Assume that the following {\sf both} conditions hold.
\begin{enumerate}
\renewcommand{\labelenumi}{\rm{(\alph{enumi})}}
\item 
$\H^i(X^\bullet)=0$ and $\H^i(Y^\bullet)=0$ for every $i\in\mathbb Z-\{0,1\}$
\item
$\Ext_A^1(\H^1(X^\bullet),\H^0(Y^\bullet))=0$.
\end{enumerate}
Then, there exist canonical homomorphisms
$$\Hom_A(\H^0(X^\bullet),\H^0(Y^\bullet))\overset{f}{\longrightarrow}
\Ext_A^2(\H^1(X^\bullet),\H^0(Y^\bullet))$$ and 
$$
\Hom_A(\H^1(X^\bullet),\H^1(Y^\bullet))\overset{g}{\longrightarrow}
\Ext_A^2(\H^1(X^\bullet),\H^0(Y^\bullet))$$
and further
the pull-back of $f$ and $g$ is the following:
\[
\begin{tikzcd}
 \Hom_{\D^{\mathrm{b}}(\mod- A)}(X^\bullet,Y^\bullet)\ar[r, dashed, "\varphi"] \arrow[d, dashed, "\psi"] 
&\Hom_A(\H^1(X^\bullet),\H^1(Y^\bullet)) \ar[d, "g"] \\
\Hom_A(\H^0(X^\bullet),\H^0(Y^\bullet)) \ar[r, "f"] 
&\Ext_A^2(\H^1(X^\bullet),\H^0(Y^\bullet)).
\end{tikzcd}
\]
\end{Lemma}
\begin{proof}
These follow from \cite[%\S6.6.2 on pp.600--601, and 
Proposition 6.6.2]{Zim14}.
\end{proof}

\begin{Lemma}\label{F(simple)}
Suppose that $A$ and $B$ are finite-dimensional self-injective $k$-algebras such tha
$A/\rad(A)$ and $B/\rad(B)$ are both separable
(see \cite[Definition 1.16.10]{Lin18}). %and Proposition 1.16.21]{Lin18}).
Further suppose that
an indecomposable $(A,B)$-bimodule $M$ is projective as a left $A$- and also as a right $B$-
module, and that $M$ induces a stable equivalence of Morita type between
$A$ and $B$.   
Then, for any simple $A$-module $S$, $S\otimes_AM$ is indecomposable
as a $B$-module.  
\end{Lemma}
\begin{proof} Follows by \cite[Proposition 4.14.9]{Lin18}.%[Theorem 2.1(ii)]{Lin96a}.
\end{proof}

\begin{Lemma}\label{GreenHeart}
Suppose that $G\geq P$ such that $P$ is a $p$-subgroup, and let $H\leq G$ with $H\geq N_G(P)$.
Then, we can define the Green correspondence 
$f:=f_{(G,P,H)}$ with respect to $(G,P,H)$ (see \cite[Chap.4 (4.2) and Theorem 4.3]{NT88}).
Set $ \mathfrak A:=\mathfrak A(G,P,H)$ as in \cite[Chap.4 (4.1)]{NT88}. 
Further assume that $Q\leq P$.
We use the notation $\Omega^G_Q(M)$ for $M\in\mod-kG$ to emphasize that 
the relatively $Q$-projective Heller functor
$\Omega^G_Q$ is taken for a $kG$-module.
Assume that $X\in\mod-kG$ is indecomposable with vertex in $\mathfrak A$.
\begin{enumerate}
\renewcommand{\labelenumi}{\rm{(\roman{enumi})}}
\item
Then $\Omega^G_Q(X)$ is still an indecomposable $kG$-module and 
this has the same vertex as $X$.
\item
$\Omega^H_Q(fX)\cong f(\Omega^G_Q(X))$ in $\mod-kH$.
\end{enumerate}
\end{Lemma}
\begin{proof}
(i) Similar to the proof of \cite[Chap.4 Theorem 4.10(i)]{NT88}.

(ii) Though the proof of (ii) is also similar to that of \cite[Chap.4 Theorem 4.10(ii)]{NT88},
we describe it in detail here.
First, $P_Q(X){\downarrow}_H=P_Q(X{\downarrow}_H)\oplus Y$
for a $kH$-module $Y$ that is relatively $Q$-projective.
Then it is easily seen that 
\begin{equation}\label{Omega_Q}
\Omega^G_Q(X){\downarrow}_H=\Omega^H_Q(X{\downarrow}_H)\oplus Y.
\end{equation}
By the definition of the Green correspondence,
$fX\,|\,X{\downarrow}_H$, so that $\Omega^H_Q(fX)\,|\,\Omega^H_Q(X{\downarrow}_H)$.
Hence (\ref{Omega_Q}) implies that
\begin{equation}\label{GreenOmega1}
\Omega^H_Q(fX)\oplus Y\,\Big|\,\Omega^H_Q(X{\downarrow}_H)\oplus Y=\Omega^G_Q(X){\downarrow}_H.
\end{equation}
On the other hand, since we can define the Green correspondent $f(\Omega^G_Q(X))$ by (i),
\begin{equation}\label{GreenOmega2}
\Omega^G_Q(X){\downarrow}_H=f(\Omega^G_Q(X))\oplus Z\text{ for a }kH\text{-module }Z,
\text{ that is }\mathfrak Y\text{-projective}
\end{equation}
where $\mathfrak Y:=\mathfrak Y(G,Q,H)$ as in \cite[Chap.4 (4.1)]{NT88}.
Thus, (\ref{GreenOmega1}) and (\ref{GreenOmega2}) yield that
$$
f(\Omega^G_Q(X))\oplus Z = \Omega^H_Q(fX)\oplus Y\oplus W
\text{ for a }kH\text{-module }W.
$$
Hence (i) and the definition of the Green correspondence imply the assertion.
\end{proof}

\begin{Lemma}\label{ProjCover}
Let $X$ be any $kG$-module.
\begin{enumerate}
\renewcommand{\labelenumi}{\rm{(\roman{enumi})}}
\item
For a projective $kG$-module $\mathfrak P$, we can consider that
$P(X\oplus\mathfrak P)=P(X)\oplus\mathfrak P$ (genuine equal).
Now, if there is a $kG$-epimorphism $\pi:P(X\oplus\mathfrak P)\twoheadrightarrow X\oplus\mathfrak P$,
then by replacing $P(X)$ and $\mathfrak P$ by suitable $kG$-isomorphisms we may consider that
$\pi=\pi_{\mathfrak P}\oplus id_{\mathfrak P}$ where $\pi_{\mathfrak P}$ is the epimorphism
such that $P(X)\overset{\pi_{\mathfrak P}}{\twoheadrightarrow}X$.
\item
For an injective $kG$-module $\mathfrak Q$, we can consider that
$I(X\oplus\mathfrak Q)=I(X)\oplus\mathfrak Q$ (genuine equal).
Now, if there is a $kG$-monomorphism $\iota:X\oplus\mathfrak Q\rightarrowtail I(X)\oplus\mathfrak Q$,
then by replacing $I(X)$ and $\mathfrak Q$ by suitable $kG$-isomorphisms we may consider that
$\iota=\iota_{\mathfrak Q}\oplus id_{\mathfrak Q}$ where 
$\iota_{\mathfrak Q}$ is the monomorphism
such that $Y\overset{\pi_{\mathfrak P}}{\rightarrowtail}I(Y)$.
\end{enumerate}
\end{Lemma}
\begin{proof}
Routine work by the definitions of a projective cover and an injective envelope.
\end{proof}

\pagebreak
\section{Notation used throughout}

\begin{Notation}\label{Notation1}
From now on till the end of this paper we throughout use the following notation.

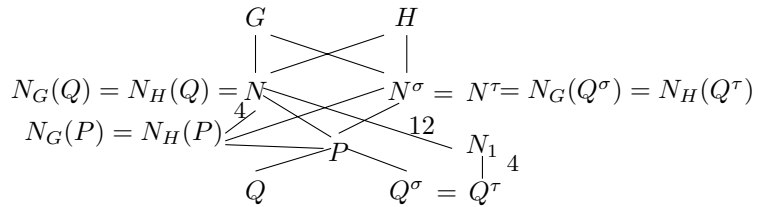
\begin{table}[h]
\vspace{-1cm}
\caption*{}
\centering
 \begin{tabular}{c |cl}
&&\\
$p$\quad & &$3$
\\
$k$\quad & &an algebraically closed field of characteristic $3$
\\
$G$&\quad&{\sf M}$_{12}$, the Mathieu group of degree $12$
\\
$H$&\quad&$\SL_3(3)=\SL(3,\mathbb F_3)$, the special linear group of size $3$ over 
$\mathbb F_3:={\mathrm{GF}}(3)$,
\\
$P$&&
$
\Big\{ 
\begin{pmatrix} 
1&0&0 \\ \gamma&1&0\\ \alpha&\beta&1
\end{pmatrix}  
\, {\Big|}\,
 \alpha, \beta, \gamma\in\mathbb F_3
\Big\}\lneqq \SL_3(3)=:H$
\\
&&
$P\cong 3_+^{1+2}$ the extra special group of order $27$ of exponent $3$
%\\
%$P$&\quad& a common Sylow $3$-subgroup of $G$ and $H$ by identifying
%\\
%$P\cong 3_+^{1+2}$&\quad& the extra special group of order $27$ of exponent $3$
\\
&&$P\in\Syl_3(G)\cap\Syl_3(H)$ by idetifying
\\
$N_G(P)=N_H(P)$
&& and further note that $N_G(P)=N_H(P)\cong P\rtimes (C_2\times C_2)$ (the Borel subgroup)
\\
$\sigma$&\quad&the outer automorphism of $P$ in $G$ with order $2$, and hence $P^\sigma=P$
\\
$\tau$&\quad&the outer automorphism of $P$ in $H$ with order $2$, and hence $P^\tau=P$,
\\
&\quad&such that $u^\sigma=u^\tau$ for every $u\in P$ 
\\
$Q$&&\quad
$\Big\{ 
\begin{pmatrix} 
1&0&0 \\ 0&1&0 \\ \alpha&\beta&1
\end{pmatrix}  
\, {\Big|}\,
 \alpha, \beta\in\mathbb F_3
\Big\}\lneqq P\lneqq\SL_3(3)=H, \ \ \ Q\cong C_3\times C_3$
\\
\\
$N$&&\quad 
$\Big\{ 
\begin{pmatrix} 
x&{\begin{matrix}0\\0 \end{matrix}}
\\ 
{\begin{matrix}\alpha&\beta\end{matrix}} & \det(x)^{-1}
\end{pmatrix}  
\, {\Big|}\,
x\in\GL_2(3), \alpha, \beta\in\mathbb F_3
\Big\}=N_H(Q) \lneqq \SL_3(3)=H
$
\\
$N_G(Q)=N_H(Q)$&& note that $P\in\Syl_3(N)$ and we can identify $N=N_H(Q)=N_G(Q)$
\\
$N_G(P)\lneqq N_G(Q)$
&&further that $N_G(P)\lneqq N_G(Q)=N=N_H(Q)\gneqq N_H(P)$, $N_G(P)=N_H(P)$
\\
&& and that
$N_G(P)/P\,C_G(P)\cong N_H(P)/P\,C_H(P)\cong C_2\times C_2$
\\
$\GL_2(3)=N_0$&& $\GL_2(3)$ such that 
$N=N_H(Q)=N_G(Q)=Q\rtimes N_0\cong(C_3\times C_3)\rtimes\GL_2(3)$
\\
$C_3\cong P_0$&& a Sylow 3-subgroup of $N_0$ and hence $P_0\cong C_3$
\\
$SD_{16}\cong L_0$&& a Sylow {\sf 2}-subgroup of $N_0\cong\GL_2(3)$, 
and hence $L_0\cong SD_{16}$
\\
$Q\rtimes SD_{16}\cong L$&& $Q\rtimes L_0$ and hence $L\lneqq N$
\\
$N_1$&&\quad
$\Big\{ 
\begin{pmatrix} 
\gamma&0&0 \\ \beta&\delta&0\\ \alpha&0&(\gamma\,\delta)^{-1}
\end{pmatrix}  
\, {\Big|}\,
\gamma,\delta\in{\mathbb F_3}^{\times}, \,\alpha,\beta\in\mathbb F_3
\Big\}
\lneqq N=N_G(Q)=N_H(Q)
$
\\
&&
$N_1\cong (C_3\times C_3)\rtimes (C_2\times C_2)
\cong\mathfrak S_3\times\mathfrak S_3$
\\&&
note that $Q^\sigma=Q^\tau\in\Syl_3(N_1)$ and $Q^\sigma=Q^\tau\,{\not=}\,Q$, in fact
\\
$
Q\,{\not=}\,Q^\sigma=Q^\tau$
&&
$
\Big\{ 
\begin{pmatrix} 
1&0&0 \\ \alpha&1&0 \\ \beta&0&1
\end{pmatrix}  
\, {\Big|}\,
 \alpha, \beta\in\mathbb F_3
\Big\}\ \cong \ C_3\times C_3
$ (compare with $Q\in\SL_3(3)=H$ above)
\\
$w$&&
$
\begin{pmatrix} 
-1&0&0 \\ 0&0&1\\ 0&1&0
\end{pmatrix}  
\in \SL_3(3)=H$
\\
&& note that
$N_1=N^w\cap N=N^w\cap N_G(Q)=N^w\cap N_G(P)$
\\
$N_{1,0}$&& $C_2\times C_2$\ \ 
such that $N_1=Q^\tau\rtimes N_{1,0}\lneqq N$
\\&&
\\
&&
$
\begin{tikzpicture}
\draw
(-1,0)node[above](G){$G$}--(-1,-0.5)node[below](N){$N$}
(1,0)node[above](H){$H$}--(1,-0.5)node[below](Nsigma){$N^\sigma$}
(2.4,-0.75)node[left](){$=\,N^\tau$}
(-1.1,-0.75)node[left](){$N_G(Q)=N_H(Q)=$}
(-1.3,-1.3)node[left](){$N_G(P)=N_H(P)$}
(-1.4,-1.3)--(-1,-1)node[left](){\small 4}
(-1.4,-1.4)--(0.7,-0.7)
(-1.4,-1.45)--(-0.1,-1.5)
(2.1,-0.75)node[right](){$= N_G(Q^\sigma)=N_H(Q^\tau)$}
(-0.8,0)--(0.7,-0.5)
(-0.8,-0.5)--(0.7,0)
(0.9,-0.9)--(0.1,-1.3)node[below](){$P$}
(-0.9,-0.8)--(0,-1.35)
(0,-1.5)--(-1,-1.8)node[below](){$Q$}
%(0.5,-2.1)node[left](){${\not=}_G$}
(0.2,-1.5)--(1,-1.8)node[below](){$Q^\sigma$}
(2.4,-2.1)node[left](){$=\,Q^\tau$}(2,-1.2)node[below](){$N_1$}
(2,-1.6)--(2,-1.9) (2.2,-1.7)node[right](){\small 4}
(1.5,-1.2)node[left](){12}
(-0.9,-0.7)--(1.6,-1.5)
%(0.5,-2.5)node[left](){${\not=}_H$}
;
\end{tikzpicture}
%\end{equation}
$
\\
&&
note that $Q$, $Q^\sigma$ are {\sf not} conjugate in $G$, and that
$Q$, $Q^\tau$ are {\sf not} conjugate in $H$
\end{tabular}
\end{table}

%\pagebreak

\begin{table}[h]
\vspace{-1cm}
\caption*{}
\centering
    \begin{tabular}{c |cl}
%\hline
&&\\
$\mathfrak M$&&
$\Sc(G\times H,\Delta P)$, the (Alperin-)Scott module
over $k(G\times H)$ with respect to $\Delta P$
\\
&&
note that $\mathfrak M$ can be defined since $N\leq G\cap H$ 
\\
&&
and also that 
%\begin{equation}\label{ScottGxH}
$\mathfrak M=\Sc(G\times H,\Delta N)
=\Sc(G\times H,\Delta(N^\sigma))$
%\end{equation}
\\&&
see \cite[Chap.4 Corollary 8.5]{NT88}, and
notice that $\Delta P\in\Syl_3(\Delta N)\cap\Syl_3(\Delta(N^\sigma))$
\\
$A$&&$B_0(kG)$, the principal block (ideal) of $kG$
\\
$B$&&$B_0(kH)$, the principal block (ideal) of $kH$
\\
$\Irr_k(A)$&&
$\{ k:=k_G, 10, 10', 15:=15_{kG}, 15^*:={15^*}_{kG}, 34, 45, 45' \}$
\\
$\Irr_k(B)$&&
$\{ k:=k_H, 3, 3^*, 6, 6^*, 7, 15:=15_{kH}, 15^*:={15^*}_{kH}\}$
\\
$\Irr_k(kN)$&&
$\{ k:=k_N, 1:=1_{kN}, 2, 2^*, 3_1, 3_2\}$
\\
&&
(note that $1_{kN}$ is a non-trivial $kN$-module of $k$-dimension one
\\
&& see
\cite[p.7 and (7.4)(ii)]{KW99a} and \cite[p.1218]{Kos87})
\\
      \end{tabular}
\end{table}
\noindent
Note that the automorphisms $\sigma$ and $\tau$ of $G$ and $H$, respectively, interchange
certain pairs of simple $A$-modules and $B$-modules above such that
\begin{equation}\label{sigma-tau}
10^\sigma=10', \ 15^\sigma=15^*, \ 45^\sigma=45' \text{ \quad \ and \quad \ }
3^\tau=3^*, \ 6^\tau=6^*, \ 15^\tau=15^*.
\end{equation}
\end{Notation}

From now on till the end of this paper we shall use the notation defined in
Notation~\ref{Notation1}.

\begin{Lemma}\label{vertex}
%It holds the following:
\begin{enumerate}
\renewcommand{\labelenumi}{\rm{(\roman{enumi})}}
\item The simple $kG$-modules in $A$ but $45$ and $45'$ have $P$
as their vertices.
\item
The simple $kG$-modules $45$ and $45'$ in $A$, respectively, have $Q^\sigma$ 
and $Q$ as their vertices.
(Note that $Q$ and $Q^\sigma$ are not conjugate in $G$ and neither in $H$.)
%\item
%The simple $kH$-modules $6$ and $6*$ in $B$, respectively, 
%have $Q^\sigma$ and $Q$ as their vertices.
\end{enumerate}
\end{Lemma}
\begin{proof}
(i) and (ii) follow from \cite[Theorem (i), (ii)]{KW99b}.
%
%(ii)
\end{proof}

%%%%%\newpage
\begin{Lemma}\label{FusionSystem}
%It holds the following:
\begin{enumerate}
\renewcommand{\labelenumi}{\rm{(\roman{enumi})}}
\item $\F_P(G)=\F_P(H)\cong\F_{\Delta P}(G\times H)$.
\item
The fusion system
$\mathcal F_{\Delta P}(G\times H)$ is saturated.
\end{enumerate}
\end{Lemma}
\begin{proof}
(i) The first equality follows from \cite{RV04} (see also 
\cite[Table 3 on p.2043 and (8) in the table on p.2050]{NU09}).
Then it is easy to see from (i) that
the correspondence 
\begin{align*}
&\Hom_{\F_P(G)}(R,S)=\Hom_{\F_P(H)}(R,S)\rightarrow
  \Hom_{\F_{\Delta P}(G\times H)}(\Delta R,\Delta S),
\\
& \qquad\qquad\qquad\qquad c_g=c_h\mapsto c_{(g,h)}
\quad \text{ for }R,S\leq P\text{ and }g\in G, h\in H
\end{align*}
induces the second isomorphism.

(ii) Since $P\in\Syl_3(G)$, $\F_P(G)$ is saturated. So the assertion follows from (i).
\end{proof}

\begin{Notation}\label{notation33}
Because of the previous lemma,
from now on we can set {\color{black}$\F:=\F_P(G)=\F_P(H)\cong\F_{\Delta P}(G\times H)$}.
\end{Notation}

\begin{Lemma}\label{BrauerIndecForP}
$\mathfrak M(\Delta P){\downarrow}^{N_{G\times H}(\Delta P)}_{C_{G\times H}(\Delta P)}$
is indecomposable.
\end{Lemma}
\begin{proof}
Follows from Lemma~\ref{FusionSystem}(ii) and \cite[Lemma 4.3(ii)]{KKM11}.
\end{proof}

\begin{Lemma}\label{BrauerIndecFor<1>}
$\mathfrak M(\Delta\langle 1\rangle){\downarrow}
^{N_{G\times H}(\Delta\langle 1\rangle)}_{C_{G\times H}(\Delta\langle 1\rangle)}$
is indecomposable.
\end{Lemma}
\begin{proof}
Easy since $\mathfrak M(\langle 1\rangle)\cong \mathfrak M$ as
$(kG,kH)$-bimodules. 
\end{proof}

\begin{Lemma}\label{BrauerIndecForC3}
Suppose that $R\leq P$ with $|R|=3$.
\begin{enumerate}
\renewcommand{\labelenumi}{\rm{(\roman{enumi})}}
\item
Assume that $R=Z(P)$.
    \begin{enumerate}
    \item
It holds that    
$N_G(R)\cong P\rtimes (C_2\times C_2)\cong N_H(R)$ and 
$C_G(R)\cong P\rtimes C_2\cong C_H(R)\cong C_N(R)$,
and hence we can identify as $N_G(R)=N_H(R)$ and 
$C_G(R)=C_H(R)=C_N(R)$
since $N_G(R)=N_G(P)\lneqq N_G(Q)$ and $N_H(R)=N_H(P)\lneqq N_H(Q)$.
(Recall that we identify $N_G(Q)$ and $N_H(Q)$ in Notation~\ref{Notation1}.).
    \item
The subgroup $R=Z(P)$ is a fully $\mathcal F$-centralized subgroup of $P$. 
    \item
$\mathfrak M(\Delta R){\downarrow}^{N_{G\times H}(\Delta R)}_{C_{G\times H}(\Delta R)}$
is indecomposable.   
     \end{enumerate}
\item
Assume that $R\,{\not=}\,Z(P)$ and that
$R$ is a fully $\F$-normalized subgroup of $P$.
     \begin{enumerate}
     \item 
The subgroup $R$ is a fully $\mathcal F$-centralized subgroup of $P$.  
     \item     
$N_G(R)\cong \mathfrak A_4\times\mathfrak S_3\gneqq C_3\times\mathfrak S_3\cong N_H(R)$
and 
$$
\qquad \qquad \quad C_G(R)\cong C_3\times\mathfrak A_4\gneqq C_3\times C_3\cong Z(P)\times R = C_H(R)=C_N(R).
$$ 
    \item
$\mathfrak M(\Delta R){\downarrow}^{N_{G\times H}(\Delta R)}_{C_{G\times H}(\Delta R)}$
is indecomposable.       
     \end{enumerate}
\end{enumerate}
\end{Lemma}

\begin{proof}
(i)-(a) follows by elementary calculations.

(i)-(b) 
Take any subgroup $S\leq P$ with $S=_G R$. 
Then, since $R=Z(P)$, $|C_P(R)|=|P|\geq |C_P(S)|$, that yields that
$R$ is a fully $\F$-centralized subgroup of $P$.

(i)-(c)
First, note that $\Delta R$ is a fully $\F$-centralized subgroup of $\Delta P$
by (b) and Lemma~\ref{FusionSystem}(i). 
Note that \cite[the paragraph after Definition 6.4.3]{Lin18} and Brauer's 3rd Main Theorem
imply that
\begin{equation}\label{almostSource}
A\text{ and }B\text{ are almost source algebras of themselves.}
\end{equation}
Now, set $A_R:=B_0(k\,C_G(R))$, and then (a) yields that
$A_R=B_0(k\,C_H(R))=B_0(k\,C_N(R))$.
Further, 
since $\mathfrak M\,|\,kG\otimes_{B_0(kN)}kH$, it holds that $\mathfrak M\,|\,A\otimes_{kN}B$.
Thus,
\begin{align*}
\mathfrak M(\Delta R)&\,|\,(A\otimes _{kN}B)(\Delta R)
\\&
=A(\Delta R)\otimes_{(kN)(\Delta R)}B(\Delta R)
\quad\text{ by (b), (\ref{almostSource}) and \cite[Theorem 9.4.7]{Lin18} }
\\&
=B_0(k\,C_G(R))\otimes_{k\,C_N(R)}B_0(k\,C_H(R)) 
\\&
\text{ by Brauer's 3rd Main Theorem and \cite[Remark 5.4.13]{Lin18}}
\\&
=A_R\otimes_{A_R}A_R
\\&
=
{_{k\,C_G(R)}}{(A_R)}_{k\,C_H(R)}
\end{align*}
Notice that $A_R$ is an indecomposable $(k\,C_G(R), k\,C_G(R))$-bimodule,
and since we can consider $A_R=B_R$, $A_R$ is considered as an
indecomposable $(A_R,B_R)$-bimodule, so that
$$\mathfrak M(\Delta R){\downarrow}^{N_{G\times H}(\Delta R)}_{C_G(R)\times C_H(R)}
\cong A_R\text{\qquad as }(A_R,B_R)\text{-bimodule}
$$
since $C_{G\times H}(\Delta R)=C_G(R)\times C_H(R)$.
Namely, $\mathfrak M(\Delta R){\downarrow}^{N_{G\times H}(\Delta R)}_{C_{G\times H}(\Delta R)}$ 
is indecomposable.

\medskip
(ii)-(a):
Suppose that $S\leq P$ and $S=_GR$. Obviously
$N_P(R)=C_P(R)=R\times Z(P)\cong C_3\times C_3$, and hence $|C_P(R)|=9$.
If $|C_P(S)|\,{\not\leq}\,9$, then $|C_P(S)|=27$, which implies that $S\leq Z(P)$,
and hence $|N_P(S)|=|P|=27>9=|N_P(R)|$, that means that $R$ is not fully $\F$-normalized in $P$,
a contradiction.

%%\medskip
(ii)-(b): 
By elementary calculations, we know the structures of $N_G(R), N_H(R), C_G(R), C_H(R)$ 
and $C_N(R)$ (cf.\cite[pp.13 and 33]{Atlas}).
Recall that we already can consider that $N\leq G\cap H$. So, the assertion follows.

%%\medskip
(ii)-(c): It holds that
$B_0(k\,C_G(R))\cong k\,C_G(R)/O_{3'}(C_G(R))\cong k(C_3\times C_3)
\cong k\,C_H(R)=k\,C_N(R)$.  
Therefore, just as in the proof of (i)-(c) above, the assertion follows.
\end{proof}

\begin{Lemma}\label{BrauerIndecForC3xC3}
Suppose that $R\leq P$ with $|R|=9$.
\begin{enumerate}
\renewcommand{\labelenumi}{\rm{(\roman{enumi})}}
\item
One of the following cases happens:
    \begin{enumerate}
    \item
$N_G(R)\cong N_H(R)\cong R\rtimes\GL_2(3)$
and $C_G(R)=C_H(R)=C_N(R)=R$.
    \item
$N_G(R)\cong N_H(R)\cong  R\rtimes\mathfrak S_3\cong P\rtimes C_2$, and 
$C_G(R)=C_H(R)=C_N(R)=R$.
    \end{enumerate}
\item
The subgroup $R$ is a fully $\F$-centralized subgroup of $P$.
\item
$\mathfrak M(\Delta R){\downarrow}^{N_{G\times H}(\Delta R)}_{C_{G\times H}(\Delta R)}$
is indecomposable.
\end{enumerate}
\end{Lemma}

\begin{proof}
(i) 
There are precisely four maximal subgroups $R$\,s of $P$ of order $9$. Two of them
are swapped by the outer-automorphism $\sigma\in\Out(G)$ (resp. $\tau\in\Out(H)$) 
if the maximal subgroups $R$\,s are considered as in $G$ (resp. $H$). 
So the assertion follows from elementary calculations.

(ii) Take any subgroup $S\leq P$ with $S=_GR$. Then $S$ is at least a subgroup of $P$
of order $9$. Hence it follows from (i) that $|C_P(S)|=|S|=|R|=|C_P(R)|$. So,
$R$ is a fully $\F$-centralized subgroup of $P$.

(iii)
By (i),
$B_0(k\,C_G(R))\cong B_0(k\,C_H(R))\cong B_0(k\,C_N(R))\cong kR$
as $(kR, kR)$-bimodules. Hence the assertion follows precisely just as in the
proof of Lemma~\ref{BrauerIndecForC3}(i)--(c) by making use of 
Lemma~\ref{FusionSystem}(ii) and (ii).
\end{proof}

%\pagebreak
%%\newpage
\begin{Lemma}\label{ScottBrauerIndec}
The Scott module $\mathfrak M$ is Brauer indecomposable.
\end{Lemma}
\begin{proof}
Set $\mathcal G:=G\times H$ and $\mathcal P:=\Delta P$,
then by Notation~\ref{notation33}, we can set
$\F:=\F_{\mathcal P}(\mathcal G)$.
We claim that for any $\mathcal Q\leq\mathcal P$,
${\mathfrak M(\mathcal Q)}{\downarrow}^{N_{\mathcal G}(\mathcal Q)}_{C_{\mathcal G}(\mathcal Q)}$
is indecomposable.
This follows from 
Lemmas~\ref{BrauerIndecForP}, \ref{BrauerIndecFor<1>}, 
\ref{BrauerIndecForC3}(i)(c), \ref{BrauerIndecForC3}(ii)(c), \ref{BrauerIndecForC3xC3}(iii)
and \cite[Theorem 1.3]{IK17}.
%Since $\mathfrak M(\Delta\langle 1\rangle)=\mathfrak M$ and 
%${C_{\mathcal G}}(\Delta\langle 1\rangle)=\mathcal G$, one can assume that $\mathcal Q\,{\not=}\,1$.
%Further, since $\F$ is saturated by Lemma~\ref{FusionSystem}(ii)
%and since $\mathfrak M=\Sc(\mathcal G, \mathcal P)$, it holds that
%$\mathfrak M(\mathcal P){\downarrow}^{N_{\mathcal G}(\mathcal P)}_{C_{\mathcal G}(\mathcal P)}$
%is indecomposable by \cite[Lemma~4.3(ii)]{KKM11}. Thus we can assume that $|\mathcal Q|=3$ or $9$.
\end{proof}

\begin{Lemma}\label{stableEqA-B}
\begin{enumerate}
\renewcommand{\labelenumi}{\rm{(\roman{enumi})}}
\item
The Scott module $\mathfrak M$ induces a
stable equivalence of Morita type between $A$ and $B$ that is realized by the functor
$F: \underline{\mathrm{mod}}\text{-}A\rightarrow\underline{\mathrm{mod}}\text{-}B$ defined that
for each {\sf indecomposable} $A$-module $X$,  
$$X\mapsto X\otimes_A\mathfrak M :=F(X)\oplus{\mathrm{(proj)}}$$ such that 
$F(X)$ is an {\sf indecomposable} $B$-module.
Namely $F$ realizes a splendid stable equivalence of Morita type between $A$ and $B$
(see \cite[Definition 9.8.1]{Lin18} and note that
$S\otimes_A\mathfrak M$ is already an {\sf indecomposable} $B$-module and hence
$S\otimes_A\mathfrak M=F(S)$ for a simple
$A$-module $S$ by \cite[Proposition 1.16.21]{Lin18} and Lemma~\ref{F(simple)}).
\item
The functor $F$ preserves vertices, and further if $X$ is a $p$-{\sf permutation} $kG$-module in $A$
then $F(X)$ is  a $p$-{\sf permutation} $kH$-module in $B$.
\item
{\color{black}
For an indecomposable $A$-module $X$,
\[ F(X)=
\begin{cases}
{f_{(H,P,N)}}^{-1}\circ f_{(G,P,N)}(X)\quad\text{ if }P\text{ is a vertex of }X, \\
{f_{(H,Q,N)}}^{-1}\circ f_{(G,Q,N)}(X)\quad\text{ if }Q\text{ is a vertex of }X.
\end{cases}
\]
}
\item
The functor $F$ commutes with the outer-automorphisms of $P$ in $G$ and $H$, that is
$$ F(X^\sigma)=F(X)^\tau \ \ \text{ for }X\in\mod-A.$$
\end{enumerate}
\end{Lemma}
\begin{proof}
(i)
First, $\F_P(G)=\F_P(H)$ by Lemma~\ref{FusionSystem}(i).

Next, we claim that for any non-trivial subgroup $R\leq P$,
$\mathfrak M(\Delta R)$ realizes a Morita equivalence between $B_0(k\,C_G(R))$ and
$B_0(k\,C_H(R))$.
Actually all the cases treated in 
Lemmas~\ref{BrauerIndecForP}, \ref{BrauerIndecForC3} and \ref{BrauerIndecForC3xC3}
have a situation such that
$\mathfrak M(\Delta R)\cong B_0(k\,C_G(R))\cong B_0(k\,C_N(R))\cong B_0(k\,C_H(R))$
as $(B_0(k\,C_G(R)), B_0(k\,C_H(R))$-bimodules. Hence the gluing method
\cite[Theorem 9.8.2]{Lin18} implies the assertion.
(ii)--(iv) are easy by Notation~\ref{Notation1}.
\end{proof}

\begin{Notation}\label{GreenCorr}
We shall use the following notation till the end of this paper;
\begin{align*}
&\mathfrak f:=f_{(G,{\color{black}P},N)}, \ \mathfrak f':=f_{(H,{\color{black}P},N)}, 
\\
&f:=f_{(G,{\color{black}Q},N)}\text{ and }f':=f_{(H,{\color{black}Q},N)}, 
\quad\text{ see \cite[Chap.4 Theorem 4.3]{NT88}).}
\end{align*}
\end{Notation}

\begin{Notation}\label{Notation2}
The picture below may be useful for understanding what's going on
(see Notation~\ref{Notation1}).
\[
\begin{tikzpicture}
\path (0,0)--node[midway,above]{$f\ \  \ \ \ \ \ \mathfrak f$}(-8,-1);
\draw
(-5,0)node[left](){$G:={\sf M}_{12}$}
    [->,thick](-5,0)--(-3,-1)
node[below](){\qquad\qquad\qquad\qquad\qquad\qquad\qquad
         $N:=N_G(Q)=N_H(Q)=Q\rtimes N_0\cong (C_3\times C_3)\rtimes\GL_2(3)$};
\draw (2,0)node[right](){$H:=\SL_3(3)$}[->,thick]--(1,-1)
    %[->,thick]--(8,-1)
%\draw (6,0)node[right](){\f'$}
;
\path (4,0)--node[midway,thick,above]{$\mathfrak f'$ \ \ \ \ \ \ \ $f'$}(-1.3,-1.3)
;
\draw[thick](-1,-1.5)--(-1,-2.5)node[below](){$N_G(P)=N_H(P)\cong P\rtimes (C_2\times C_2)$}
;
\end{tikzpicture}
\]
From now on till the end of paper, we use the notation $F$ in the sense of
the previous lemma.
\end{Notation}

%\pagebreak
\section{PIMs over $A$ and $B$}
\begin{Lemma}\label{P(6)-P(15)} 
%For $15,15^*,6,6^*\in\Irr_k(B)$,
%The following holds;
\begin{enumerate}
\renewcommand{\labelenumi}{\rm{(\roman{enumi})}}
\item
\item[{}]
\begin{tikzpicture}
\draw (-1.5,-2.6)node[left](){$P(15)$=};
\draw 
%(-2,0)node[left](P(15)){$P(15)=$}
(0,0) node[above](15){$15$}--(0,-0.3) node[below](k){$k$}(0,-0.7)--(0,-1)node[below](7){$7$}
%(0,-1.4)--(1,-2,4);                             
(0,-1.4)--(1,-2.4)node[below](3){$3$}
(-0.1,-1.4)--(-0.4,-1.7)node[below left=0.3pt](k){$k$}
(-0.8,-2.2)--(-1.0,-2.4)node[below](15){$15$}
(-0.4,-2.2)--(-0.2, -2.4)node[below](15*){\ $15^*$}
(-0.8,-2.8)--(-0.6,-3.0)node[below](k){\ $k$}
(-0.2,-2.8)--(-0.4,-3.0)
(-0.3,-3.4)--(0,-3.7)node[below](7){\ \ $7$}
(1,-2.9)--(0.2,-3.7)
(0.1,-4.1)--(0.1,-4.4)node[below](k){$k$}
(0.1,-4.8)--(0.1,-5.1)node[below](15){$15$}
;
\end{tikzpicture}
\qquad\qquad
\begin{tikzpicture}
\draw (-1.5,-2.6)node[left](){$P(15^*)$=};
\draw 
%(-2,0)node[left](P(15)){$P(15)=$}
(0,0) node[above](15){$15^*$}--(0,-0.3) node[below](k){$k$}(0,-0.7)--(0,-1)node[below](7){$7$}
%(0,-1.4)--(1,-2,4);                             
(0,-1.4)--(1,-2.4)node[below](3){$3^*$}
(-0.1,-1.4)--(-0.4,-1.7)node[below left=0.3pt](k){$k$}
(-0.8,-2.2)--(-1.0,-2.4)node[below](15){$15^*$}
(-0.4,-2.2)--(-0.2, -2.4)node[below](15*){\ $15$}
(-0.8,-2.8)--(-0.6,-3.0)node[below](k){\ $k$}
(-0.2,-2.8)--(-0.4,-3.0)
(-0.3,-3.4)--(0,-3.7)node[below](7){\ \ $7$}
(1,-2.9)--(0.2,-3.7)
(0.1,-4.1)--(0.1,-4.4)node[below](k){$k$}
(0.1,-4.8)--(0.1,-5.1)node[below](15){$15^*$}
;
\end{tikzpicture}

\item
\item[{}]
%The PIMs $P(6)$ and $P(6^*)$ have the following structures;\\
%%%\medskip
%\qquad
%
%\boxed{
\begin{tikzpicture} 
\draw (-1.7,-1.9)node[left](){$P(6)$=};
\draw
(0,0) node[above](){$6$}--(0.3,-0.3)node[below](3){\quad $3$}
(0.6,-0.7)--(0.9,-1)node[below](7){$7$}
(-0.1,0)--(-1.1,-1)node[below](){\ $6^*$}
(-1.1,-1.4)--(-1.1,-1.7)node[below](){$6$}
(0.4,-0.7)--(-0.9,-2)
(0.9,-1.4)--(0.9,-1.7)node[below](){$k$}
(0.7,-1.3)--(0.2,-1.7)node[below](){$3^*$ \ }
(-1.1,-2.1)--(-1.1,-2.4)node[below](){\ \ $6^*$ \, }
(-0.9,-2.6)--(-0.1,-2)
(0.9,-2.1)--(0.9,-2.4)node[below]{$7$}
(-1.1,-2.8)--(-0.1,-3.9)node[below]{$6$}
(0.7,-2.9)--(0.4,-3.3)node[below]{$3$}
(0.1,-4)--(0.3,-3.7)
(-0.9,-2.1)--(0.2,-3.3)
(-1,-1.3)--(-0.1,-1.9)
(0.2,-2.1)--(0.8,-2.6)
;
\end{tikzpicture}
\qquad\qquad
\begin{tikzpicture}
\draw (-1.7,-1.9)node[left](){$P(6^*)$=};
\draw 
(0,0) node[above](){$6^*$}--(0.3,-0.3)node[below](3){\quad $3^*$}
(0.6,-0.7)--(0.9,-1)node[below](7){$7$}
(-0.1,0)--(-1.1,-1)node[below](){\ $6$}
(-1.1,-1.4)--(-1.1,-1.7)node[below](){\ $6^*$}
(0.4,-0.7)--(-0.9,-2)
(0.9,-1.4)--(0.9,-1.7)node[below](){$k$}
(0.8,-1.3)--(0.2,-1.7)node[below](){$3$ \ }
(-1.1,-2.1)--(-1.1,-2.4)node[below](){\ \ $6$ \, }
(-0.9,-2.6)--(-0.1,-2)
(0.9,-2.1)--(0.9,-2.4)node[below]{$7$}
(-1.1,-2.8)--(-0.1,-3.9)node[below]{$6^*$}
(0.7,-2.9)--(0.4,-3.3)node[below]{$3^*$}
(0.1,-4)--(0.3,-3.7)
(-0.9,-2.1)--(0.2,-3.3)
(-1,-1.3)--(-0.1,-1.9)
(0.2,-2.1)--(0.8,-2.6)
;
\end{tikzpicture}
\end{enumerate}
\end{Lemma}
\begin{proof}
(i) and (ii) follow from \cite[Theorem 1]{Kos87}.
\end{proof}

%%%%\newpage
\section{The images $F(k_G), F(15_{kG}), F({15^*}_{kG})\text{ and }F(34)$}

\begin{Lemma}\label{F(k)F(15)F(15*)}
$F(k_G)=k_H, \  F(15_{kG})=15_{kH}, \  F({15^*}_{kG})={15^*}_{kH}.$
\end{Lemma}
\begin{proof}
First, since $P$ is a vertex of $k_G$, Lemma~\ref{stableEqA-B}(iii) yields that
\[ F(k_G)={\mathfrak f'}^{-1}\circ \mathfrak f(k_G)={\mathfrak f'}^{-1}(k_N)=k_H.
\]
By \cite[Theorem (i)]{KW99b}, $P$ is a vertex of $15_{kG}$.
Thus
\begin{align*}
\mathfrak f(15_{kG})&=15_{kG}{\downarrow}{^G_N}
\text{\ \qquad from \cite[(6.2)]{KW99b} }% and the definition of a Green correspondence} 
\\&\cong{15_{kH}}{\downarrow}^H_N=\mathfrak f'(15_{kH})
\text{\ \qquad by \cite[(1.3)Proposition (iv)]{Kos87}}.
\end{align*}
On the other hand, Lemma~\ref{stableEqA-B}(iii) implies that
${\mathfrak f'}^{-1}\circ{\mathfrak f}(15_{kG})\cong F(15_{kG})$, and hence the second equality holds.

Finally,
\begin{align*}
F(15^*)&= F(15^\sigma) \text{ \ \ by (\ref{sigma-tau})}
\\
&=F(15)^\tau \text{ \ \ by Lemma~\ref{stableEqA-B}(iv)}
\\
&={15_{kH}}^\tau \text{ \ \ by the second equality}
\\
&=
{15_{kH}}^* \text{ \ \ by (\ref{sigma-tau})}.
\end{align*}
\end{proof}

\begin{Lemma}\label{kN-module}
There exists uniquely up to isomorphism the $kN$-module of the form
$
U(2^*,3_2,2).
$
\end{Lemma}
\begin{proof} Follows from \cite[Theorem 3.8]{Kos86}.
%\\
%{\sf\small\color{red} (The second author warned the first author 
%that we need more detailed and convinced proof
%because there might be infinitely many non-isomorphic uniserial modules of length $3$
%that have the SAME Loewy structure, though. For it the first author can provide
%another a few pages for its proof.)}
\end{proof}

\begin{Lemma}\label{F(34)} $F(34)=7$.
\end{Lemma}
\begin{proof}
First notice that $P$ is a vertex of $34$ by \cite[Theorem (i)]{KW99b}.
Now, it follows from Lemma~\ref{kN-module}, \cite[(3.7)]{KW99b} 
and \cite[(1.3)Proposition(iv)]{Kos87}
that 
$$\Omega^0(34{\downarrow}^G_N)\cong
U(2^*,3_2,2)\cong7{\downarrow}^H_N.
$$
So that $\mathfrak f'(7)=\mathfrak f(34)$ by the definition of Green correspondence.
Thus, it follows by Lemma~\ref{stableEqA-B}(iii) that
$F(34)={\mathfrak f'}^{-1}\circ \mathfrak f(34)={\mathfrak f'}^{-1}\circ \mathfrak f'(7)\cong 7$.
\end{proof}

%%%%\newpage
\section{The images $F(10)$ and $F(10')$}

In \S\S 6--7 we freely use \cite[(1.1)Lemma]{Kos85}, that is quite useful to know precise
Loewy and socle structures for $kG$-modules.

\begin{Lemma}\label{RelativeProjCover}
\begin{enumerate}
\renewcommand{\labelenumi}{\rm{(\roman{enumi})}}
\item
${\mathrm{Sc}}(N,Q)\cong U(k,1,k)%\boxed{\begin{matrix} k \\ 1 \\ k\end{matrix}}
\ \leftrightarrow \chi_I+\chi_{2a}$\\
where $\chi_I, \chi_{2a}\in\Irr(N)$ and the indices are as in \cite[(7.4)(ii)]{KW99b}.
%%%\smallskip
\item 
${\mathrm{Sc}}(N,N_0)={\mathrm{Sc}}(N,P_0)\cong
U(k,2,3_1,2^*,k)
%\boxed{\begin{matrix} k\   \\ 2\ \\3_1\\2^*\\ k\  \end{matrix}}
\ \leftrightarrow \chi_I+\chi_{8_1}$ 
\\
where the indices of $\chi_i$ are as in {\rm{(i)}}.
\end{enumerate}
\end{Lemma}
\begin{proof}
(i)
From Notation~\ref{Notation1}, $SD_{16}\in\Syl_2(N_0)$. % (note that $N_0\cong\GL_2(3)$).
Thus, by \cite[Chap.4 Corollary 8.5]{NT88} and \cite[Lemmas~3.1 and 3.4(i)]{Kos86},
%\begin{align*}
\[
\Sc(N,Q)\cong\Sc(N,Q\rtimes SD_{16})\,\Big|\,
k_{Q\rtimes SD_{16}}{\uparrow}^N=
k_{Q\rtimes SD_{16}}{\uparrow}^{Q\rtimes\GL_2(3)}
%\\
=k_{SD_{16}}{\uparrow}^{\GL_2(3)}
\cong
U(k,1,k).
\]
%\end{align*}
Since the final term is considered as an indecomposable $kN$-module,
we have
\begin{equation*}\label{ScottForQ} 
\Sc(N,Q)\cong U(k,1,k).%\boxed{\begin{matrix} k \\ 1\\ k\end{matrix}}.
\end{equation*}
Note that $\Sc(N,Q)$ is a trivial source $kN$-module. Hence
by \cite[Chap.4, Theorem~8.9(iii)]{NT88}, %and elementary calculations,
we know that the corresponding ordinary characters are as it is desired.

(ii) As in the proof of (i),
\cite[Chap.4 Theorem 8.5]{NT88} and \cite[Lemmas 3.1 and 3.4(ii)]{Kos86} imply that
\begin{equation*}\label{ScottForC3}
\Sc(N,P_0)\cong\Sc(N,N_0)\,\Big|\,k_{N_0}{\uparrow}^N\cong 
k_{\GL_2(3)}{\uparrow}^{(C_3\times C_3)\rtimes\GL_2(3)}
\cong
U(k,2,3_1,2^*,k).
\end{equation*}
For the ordinary characters corresponding to $\Sc(N,N_0)$, just as in the proof of (i).
\end{proof}

\begin{Lemma}\label{InductionFromN1}
Let $1_{N_1}\in\Irr(N_1)$ be the trivial ordinary character of $N_1$
(see Notation~\ref{Notation1}). 
Note that $N^\tau\,{\not=}\,N$, but $Q^\tau\lneqq P\lneqq N\cap N^\tau$
(see the diagram in Notation~\ref{Notation1}).
Then 
\begin{enumerate}
\renewcommand{\labelenumi}{\rm{(\roman{enumi})}}
\item
$1_{N_1}{\uparrow}^N=\chi_I+\chi_{3_1}+\chi_{8_1}$ (as ordinary characters),
where $\chi_i$ is the same as in \cite[(7.4)(i)]{KW99b}.
\item
$k_{N_1}{\uparrow}^N = 2\times k_N+2_{kN}+{2^*}_{kN}+2\times 3_1
\text{ (as composition factors)}.
$
%
%%%%%\newpage
\item
\item[{}]
\begin{tikzpicture}
\draw
(-1.3,-0.7)node[left]()
{${\mathrm{Sc}}(N,Q^\tau)={\mathrm{Sc}}(N,N_1)=k_{N_1}{\uparrow}^N
%=
%\begin{matrix}
%\boxed{\begin{matrix} k \ \ 3_1 \\ 2^*\\ 2 \ \ k \\ 3_1
%       \end{matrix}}.
%\\
%{\rm{(Loewy\ series)}}
%\end{matrix}
=$}
(-1,0)node[above](){$k$}--(-0.2,-1.1)node[below](){\  \, $2$}
(0,0)node[above](){$3_1$}--(0,-0.3)node[below](){$2^*$}
(0,-0.7)--(0,-1)
(0,-1.5)--(0,-1.8)node[below](){$3_1$}
(0.2,-0.6)--(1.1,-1.8)node[below](){$k$}
;
\end{tikzpicture}
\item
$\mathcal H_{Q^\tau}(k_N)=
U(3_1,2^*,2,3_1)$
and
$\mathcal H_Q(k_N)=
U(3_2,2,2^*,3_2).$
\end{enumerate}
\end{Lemma}
\begin{proof}
(i) Easy by the definitions of $N$, $N_1$ in Notation~\ref{Notation1}.

(ii) Follows by (i) and \cite[(7.4)(ii)]{KW99b}.

(iii)
Set $Y:={k_{N_1}}{\uparrow}^N$.
First, since $P\in\Syl_3(N)$ is a vertex of $k_N$ by \cite[Corollary~7.6]{NT88},
$k_N$ cannot be relatively $Q^\tau\in\Syl_3(N_1)$-projective. Hence
\begin{equation}\label{k_N}
k_N\,{\not|}\,Y.
\end{equation}

Next we prove 
\begin{equation}\label{DirectSummand3_1}
3_1\,{\not|}\,Y.
\end{equation}
Suppose that $3_1\,|\,Y$. Then there exists a $U_{kN}$ such that $Y=3_1\oplus U$.
Then, by looking at the composition factors of $U$ by (ii),
it follows from (\ref{k_N}) and \cite[Chap.4 Theorem 8.9(iii)]{NT88} that
$U$ is indecomposable. By (i) and \cite[Chap.4 Theorem 8.9(iii)]{NT88},
$U$ affords an ordinary character $\chi_U:=\chi_1+\chi_{8_1}$. 
Then, $\chi_U(3a)=\chi_U(3c)=0$ and $\chi_U(3b)=3$ where $3a, 3b, 3c$ are elements of $N$
of order $3$, respectively such that $|C_N(3a)|=54, |C_N(3b)|=18, |C_N(3c)|=9$, respectively
(see \cite[(7.4)(i)]{KW99b}. So that we can consider that 
$Q^\tau=\langle 3a\rangle\times\langle z\rangle$ with $z\in Z(P)$ 
(recall that $Q^\tau\leq P^\tau=P\geq Q$ in Notation~\ref{Notation1}) and 
that $P=Q^\tau\rtimes\langle 3b\rangle$.
Now \cite[II Lemma 12.6(iii)]{Lan83} implies that 
$3b\in R$ for a vertex $R$ of $U$ and that neither $3a$ nor $z$ is contained
in any vertex of $U$. However, since $U$ is relatively $N_1$-projective (and hence 
relatively $Q^\tau$-projective), $3b\in R\leq_N Q^\tau$. This means that
$3b\in_N Q^\tau-\{1\}=\{(3a)^n\,|\,n\in N\}\ni 3a$, and hence $3b=_N 3a$, a contradiction.

Next, assume that $Y$ has a direct summand $U$ such that $U\,|\,Y$ and that $U$ is affords the 
character $\chi_{8_1}\in\Irr(N)$ (see (i) and \cite[Chap.4 Theorem 8.9(iii)]{NT88}), and hence
$Y=U\oplus V$ for a $kN$-submodule $V$ of $Y$. 
Then, again (i), (ii) and \cite[Chap.,4 Theorem 8.9(iii)]{NT88} imply that the liftable ordinary character 
afforded by $V$ is $\chi_1+\chi_{3_1}$, which yields that $V=k+3_1$ (as composition factors).
Since $\Ext_{kN}^1(k,3_1)=\Ext_{kN}^1(3_1,k)=0$ by \cite[Theorem 3.8]{Kos86},
$V=k\oplus 3_1$, that contradicts (\ref{k_N}). 
Thus, $Y$ is indecomposable, and hence by \cite[Chap.4 Theorem~8.4 and Corollary~8.5]{NT88}, 
\begin{equation}\label{YisScott}
Y:={k_{N_1}}{\uparrow}^N=\Sc(N,N_1)=\Sc(N,Q^\tau).
\end{equation}
Next, it follows from \cite[I Corollary 17.4]{Lan83} that $Y$ has a submodule isomorphic to $3_1$
by \cite[(7.4)(ii)]{KW99b}. That is, by noting the self-dualities, we have
\begin{equation}\label{31<Y}
[Y,3_1]^N=[3_1,Y]^N\geq 1.
\end{equation}
Next suppose that $[Y,3_1]^N\,{\not=}\,1$. Then by (\ref{31<Y}) and (ii), 
$[Y,3_1]^N=[3_1,Y]^N=c_Y(3_1)=2$. Thus $(3_1\oplus 3_1)\,|\,Y$ which contradicts 
(\ref{DirectSummand3_1}). So,
\begin{equation}\label{31-Y}
[Y,3_1]^N=[3_1,Y]^N=1.
\end{equation}
Obviously, it follows from \cite[Chap.4 Theorem 8.9(i)]{NT88} and (i) that
\begin{equation}\label{k_N-Y}
[Y,k_N]^N=[k_N,Y]^N=1.
\end{equation}
Now we claim that $[Y,2]^N=0$.  We use the notation in \cite[(7.4)(iii)]{KW99b}. That is,
the simple $2_{kN}$ affords the ordinary character $\chi_2$. Then, elementary calculations 
imply that $[2{\downarrow}_{N_1}, k_{N_1}]^{N_1}=[k_{N_1},2{\downarrow}_{N_1}]^{N_1}=0$.
Hence
\begin{align*}
[Y,2]^N&=[{k_{N_1}}{\uparrow}^N, 2]^N \text{\quad by (\ref{YisScott})}
\\
&=[({k_{N_1}}{\uparrow}^N)\otimes_k 2^*, k_N]^N \text{\quad by the adjointness}
\\
&=[(k_{N_1}\otimes_k 2^*{\downarrow}_{N_1} ){\uparrow}^N, k_N]^N
\text{ by \cite[II Corollary 6.3]{Lan83}}
\\
&=
[2^*{\downarrow}_{N_1}, k_{N_1}]^{N_1} \text{\quad by Frobenius reciprocity} 
\\
&=[k_{N_1},2{\downarrow}_{N_1}]^{N_1}  \text{\quad by the adjointness}
\\
&=0 \text{\quad by the above}.
\end{align*}
Similarly we have $[2,Y]^N=0$. Thus, as usual the adjointness yields that 
\begin{equation}\label{2inY}
[Y,2]^N=[2^*,Y]^N=[2,Y]^N=[Y,2^*]^N=0
\end{equation}
and hence by (\ref{k_N-Y}), (\ref{2inY}), (ii) and the self-duality of $Y$,
\begin{equation}\label{L1(Y)}
L_1(Y)= k\oplus 3_1 \cong S^1(Y):=\soc(Y).
\end{equation}
Now, let $W$ be the trivial source uniserial $kN$-module in Lemma~\ref{RelativeProjCover}(ii)
affording the ordinary character $\chi_1+\chi_{8_1}$. Then, as usual, 
\cite[Chap.4 Theorem~8.9(iii)]{NT88}
and (i) imply that 
\[
[Y,W]^N=[W,Y]^N=2.
\] 
Hence by (\ref{k_N-Y}) there is a homomorphism $\varphi\in\Hom_{kN}(Y,W)$ such that
\[
Y\twoheadrightarrow {\mathrm{Im}}(\varphi)\ \in \ 
\{W:=
U(k,2,3_1,2^*,k), U(2,3_1,2^*,k),U(3_1,2^*,k),U(2^*,k) 
\}.
\]
${\mathrm{Im}}(\varphi)$ cannot be the second nor the fourth uniserial modules
from (\ref{2inY}).
If ${\mathrm{Im}}(\varphi)=W$, then by the dualilties, we can consider $W\cong W^*\leq Y^*\cong Y$,
and hence $Y\cong W\oplus 3_1$ by (ii) and (\ref{31-Y}), contradicting (\ref{YisScott}). Hence
\begin{equation}\label{Im-phi}
Y\ \twoheadrightarrow\ {\mathrm{Im}}(\varphi)\cong
U(3_1,2^*,k).
\end{equation}
So that by (\ref{L1(Y)}) and (ii),
\begin{equation}\label{LoewyY}
Y = \boxed{\begin{matrix} 3_1 \ \ \ k \\ 2^* \cdots \\ k \cdots \\ \vdots  \end{matrix}  }
\text{\quad (Loewy series), \  and }2, 3_1\text{ are left}.
\end{equation}

Next, we claim that $L_2(Y)=2^*$. Suppose that $L_2(Y)\,{\not=}\,2^*$. Then, by noting that
$3_1$ does not show up in $L_2(Y)$ by the structures of $P(3_1)$ and $P(k_N)$ in
\cite[Theorem 3.8]{Kos87}, (\ref{LoewyY}) says that $L_2(Y)=2^*\oplus 2$.
Then, by the self-duality of $Y$,
\[
Y = \boxed{\begin{matrix} 3_1 \ \ \ k \\ 2^* \ \ \ 2\, \\ k \ \ \ \, \, 3_1  \end{matrix}  }
\text{ (Loewy series and socle series)}
\]
since $\Ext_{kN}^1(k,3_1)=0$ by \cite[Theorem 3.8]{Kos87}.
We freely use [Ibid.]
in the following several times.
First, $Y$ has a submodule $Y_1$ whose socle series is 
$\boxed{\begin{matrix}2^*\\ k \ 3_1\end{matrix}  }$.
Hence, $Y_1=Y_2\oplus 3_1$ for a submodule $Y_2=
U(2^*,k)$
%\boxed{\begin{matrix}2^* \\ k\, \end{matrix}}$
by looking at the $L_2(P(k_N))$. Then, 
$
\overline Y:=Y/Y_2 =
\boxed{ \begin{matrix}3_1 \ \ k \\ 2 \\ \, 3_1\end{matrix} }
\text{ (Loewy series).}
$
Thus $\overline Y$ has a submodule $\overline V$ such that $\overline Y/\overline V\cong k_N$,
which means that $\overline V$ has a uniserial submodule 
$\overline W:=
U(2,3_1)$
with $\overline V/\overline W=3_1$.
Then, by the shape of $L_2(P(3_1))$, $\overline V = 3_1\oplus\overline W$, that implies that
$\overline V$ has a submodule $\overline Z:=3_1\oplus 3_1$.
So that 
$\overline Y/\overline Z\cong
U(k,2)$.
Then, there is submodules $Z$ $Y$ such that
$$
Y \gneqq Z \gneqq Y_2
\text{\quad such that \quad}
Y/Z\cong
U(k,2),
Z/Y_2\cong 3_1\oplus 3_1 = \overline Z
\text{ and }
Y_2\cong 
U(2^*,k).
$$
Let us look at $Z$. It follows from (\ref{L1(Y)}) and the socle series of $Y$ above
and by noting that $Z\leq Y$, $3_1\,|\,S_1(Z)$ and $3_1\,|\,S_3(Z)$.
This yields that
$Z=3_1\oplus Y_3$ for a submodule $Y_3$ with $Y_3=
U(3_1,2^*,k)$.
%\boxed{\begin{matrix} 3_1\\2^*\\ k \ \end{matrix}}$.
On the other hand, by (\ref{Im-phi}) and the self-duality of $Y$, $Y$ has a submodule
$Y_4\cong
U(k,2,3_1)$.
%\boxed{\begin{matrix} k\, \\ 2\, \\ 3_1\end{matrix}}$.
Since both of $Y_3$ and $Y_4$ are submodules of $Y$, we can consider $Y_3\cap Y_4$ in $Y$
and that is obviously $0$, that means that $Y=Y_3\oplus Y_4$ by (ii), contradicting
(\ref{YisScott}). Thus, we have by (\ref{LoewyY}) that
\begin{equation}\label{LoewyY-2}
Y = \boxed{\begin{matrix} 3_1 \ \ \ k \\ 2^* \\ k \cdots \\ \vdots  \end{matrix}  }
\text{\quad (Loewy series), \  and }2, 3_1\text{ are left}.
\end{equation}
Then by looking at the structure of $P(2^*)$ in [Ibid.], %\cite[Theorem 3.8]{Kos86},
\begin{equation}\label{LoewyY-3}
Y = \boxed{\begin{matrix} 3_1 \ \ \ k \\ 2^* \\ k \cdots \\ 3_1\cdots\\ \vdots  \end{matrix}  }
\text{\quad (Loewy series), \  and only }2 \text{ is left}.
\end{equation}
Thus, since $2\,{\not |}\,L_2(P(3_1))$ by [Ibid.],%\cite[Theorem 3.8]{Kos86}, 
we have
\begin{equation}\label{LoewyY-4}
Y = \boxed{\begin{matrix} 3_1 \ \ \ k \\ 2^* \\ k \cdots \\ 3_1\cdots \end{matrix}  }
\text{\quad (Loewy series), \  and only }2 \text{ is left}.
\end{equation}

If $2\,|\,L_4(Y)$, then $\Ext_{kN}^1(k,3_1)\,{\not=}\,0$, a contradiction by 
[Ibid.], %\cite[Theorem 3.8]{Kos86},
so that $2\,|\,L_3(Y)$, that is
\begin{equation}\label{LoewyY-5}
Y = \boxed{\begin{matrix} 3_1 \ \ \ k \\ 2^* \\ k \ \ \  2\\ 3_1\end{matrix}  }
\text{\quad (Loewy series)}.
\end{equation}
Thus, again by the self-dualities of $Y$ and the the existence of $Y_4$ in 
the proof of (\ref{LoewyY-5}), we get that
\begin{equation}\label{LoewyY-6}
Y = 
\begin{matrix}
\boxed{\begin{matrix} 3_1 \ \ \ k \\ 2^* \\ k \ \ \  2\\ 3_1\end{matrix}  }
\\
\text{ (Loewy series) }
\end{matrix}
= 
\begin{matrix}
\boxed{\begin{matrix}  3_1 \\ k \ \ 2^* \\ 2 \\ 3_1 \ \ k\end{matrix}  }
\\
\text{ (socle series)}
\end{matrix} 
\ \supseteq \ \ 
\boxed{\begin{matrix} k \\ 2 \\ \ 3_1\end{matrix} }.
\end{equation}
Thus again by %\cite[Theorem 3.8]{Kos86}, 
the shape of $P(3_1)$ in [Ibid.], we finally have
the picture of $Y$ as stated in (iii).

(iv) follows  immediately from (iii).
\end{proof}

\begin{Lemma}\label{H_Q(k_H)} 
%It holds that
\begin{enumerate}
\renewcommand{\labelenumi}{\rm{(\roman{enumi})}}
\item
\item[{}]
\begin{tikzpicture}
\draw
(-1.5,-0.8)node[left](){$P_Q(k_H)$=}
(-1,0.2)node[left](){$k$}
(0,0)node[above](){$6$}--(0.5,-0.4)node[below](){$6^*$}
(-0.1,-0)--(-0.5,-0.4)node[below](){$3$}
(-0.5,-0.8)--(-0.5,-1.1)node[below](){$7$}
(-0.5,-1.5)--(-0.5,-1.8)node[below](){\ $3^*$}
(-0.3,-0.7)--(0.5,-1.8)node[below](){$6$}
(0.5,-0.8)--(0.5,-1.7)
(0.3,-0.8)--(-0.3,-1.8)
(0.3,-2.2)--(0,-2.6)node[below](){$6^*$}
(-0.5,-2.3)--(-0.1,-2.6)
(-0.6,-1.5)--(-1,-2.6)node[below](){$k$}
(-1.2,0)--(-0.6,-1.1)
;
\end{tikzpicture}
\quad\qquad and \quad\qquad
\begin{tikzpicture}
\draw
(-1.5,-0.8)node[left](){$P_{Q^\tau}(k_H)$=}
(-1,0.2)node[left](){$k$}
(0,0)node[above](){$6^*$}--(0.5,-0.4)node[below](){$6$}
(-0.1,-0)--(-0.5,-0.4)node[below](){$3^*$}
(-0.5,-0.8)--(-0.5,-1.1)node[below](){$7$}
(-0.5,-1.5)--(-0.5,-1.8)node[below](){\ $3$}
(-0.3,-0.7)--(0.5,-1.8)node[below](){$6^*$}
(0.5,-0.8)--(0.5,-1.7)
(0.3,-0.8)--(-0.3,-1.8)
(0.3,-2.2)--(0,-2.6)node[below](){$6$}
(-0.5,-2.3)--(-0.1,-2.6)
(-0.6,-1.5)--(-1,-2.6)node[below](){$k$}
(-1.2,0)--(-0.6,-1.1)
;
\end{tikzpicture}

\item
\item[{}]
\begin{tikzpicture}
\draw
(-1.5,-0.8)node[left](){$\mathcal H_{Q}(k_H)$=}
    %(-1,0.2)node[left](){$k$}
(0,0)node[above](){$6$}--(0.5,-0.4)node[below](){$6^*$}
(-0.1,-0)--(-0.5,-0.4)node[below](){$3$}
(-0.5,-0.8)--
(-0.5,-1.1)node[below](){$7$}
(-0.5,-1.5)--(-0.5,-1.8)node[below](){\ $3^*$}
(-0.3,-0.7)--(0.5,-1.8)node[below](){$6$}
(0.5,-0.8)--(0.5,-1.7)
(0.3,-0.8)--(-0.3,-1.8)
(0.3,-2.2)--(0,-2.6)node[below](){$6^*$}
(-0.5,-2.3)--(-0.1,-2.6)
   %(-0.6,-1.5)--(-1,-2.6)node[below](){$k$}
   %(-1.2,0)--(-0.6,-1.1)
;
\end{tikzpicture}
\qquad\qquad and \qquad
\begin{tikzpicture}
\draw
(-1.5,-0.8)node[left](){$\mathcal H_{Q^\tau}(k_H)$=}
    %(-1,0.2)node[left](){$k$}
(0,0)node[above](){$6^*$}--(0.5,-0.4)node[below](){$6$}
(-0.1,-0)--(-0.5,-0.4)node[below](){$3^*$}
(-0.5,-0.8)--
(-0.5,-1.1)node[below](){$7$}
(-0.5,-1.5)--(-0.5,-1.8)node[below](){\ $3$}
(-0.3,-0.7)--(0.5,-1.8)node[below](){$6^*$}
(0.5,-0.8)--(0.5,-1.7)
(0.3,-0.8)--(-0.3,-1.8)
(0.3,-2.2)--(0,-2.6)node[below](){$6$}
(-0.5,-2.3)--(-0.1,-2.6)
   %(-0.6,-1.5)--(-1,-2.6)node[below](){$k$}
   %(-1.2,0)--(-0.6,-1.1)
;
\end{tikzpicture}
\end{enumerate}
\end{Lemma}
\begin{proof}
(i) Recall that $L:=Q\rtimes L_0\lneqq N$ with $L_0\cong SD_{16}$ in Notation~\ref{Notation1}.
Then it follows from \cite[(1.17)Proposition(iii)]{Kos87} that
\begin{equation}\label{P_Q(k)}
U:=k_L{\uparrow}^H 
\ = \
\begin{matrix}
\boxed{
\begin{matrix} k&&6\ \\ 3&&6^*\\7&&6\ \\ k&&3^*\\&6^*&\end{matrix}
}
\\
\text{Loewy series}
\end{matrix}
\ = \
\begin{matrix}
\boxed{
\begin{matrix} &6&\\ k&&3\ \\ 7&&6^* \\ \ 3^*&&6\  \\k&&6^*\end{matrix}
}
\\
\text{socle series}
\end{matrix}.
\end{equation}
We claim next that $U$ is indecomposable. 

So, suppose that $U=V\oplus W$ with $kH$-submodules $V\,{\not=}\,0$ and
$W\,{\not=}\,0$. We can assume that $V=\Sc(H,L)$, and hence $V=\Sc(H,Q)$ by
\cite[Chap.4 Corollary 8.5]{NT88}. 
%First since $V$ is $Q$-projective and $Q\not= 1$,
%$c_V(k)\not= 1$, and hence, $\c_V(k)=2$.
Then by (\ref{P_Q(k)}), $L_1(V)\cong k$ and $L_1(W)\cong 6$.
%Set $J:={\mathrm{rad}}(kH)$.
Then 
\begin{align*}
L_2(U)&=L_2(V)\oplus L_2(W)
\\&\subseteq L_2(P(k_N))\oplus L_2(P(6))
\\&
= (7\oplus 15\oplus 15^*)\oplus (3\oplus 6^*) \
\text{ by \cite[Theorem 1]{Kos87}}.% and Lemma \ref{P(6)-P(15)}(ii)} .
\end{align*}
Hence (\ref{P_Q(k)}) and Lemma~\ref{P(6)-P(15)}(ii) yield that
$L_2(W)=3\oplus 6^*$, which means that $L_2(V)=0$, 
and hence $V\cong k_H$, which is a contradiction since $Q \lneqq P\in\Syl_3(H)$.

Thus, $U=\Sc(H,L)=\Sc(H,Q)=P_Q(k_H)$. Therefore (\ref{P_Q(k)}) and 
the structures of $P(6)$ and $P(k_H)$ in \cite[Theorem 1]{Kos87} %Lemma~\ref{P(6)-P(15)}(ii) 
imply the first assertion.
The latter part is obtained immediately by the first part and (\ref{sigma-tau}).

(ii) Follows immediately by (i) since $\mathcal H_Q(k_H):=\Omega_Q(k_H)/k_H$.
\end{proof}

%%%%\newpage
\begin{Lemma}\label{Res_N(10')}
%\begin{enumerate}
%\renewcommand{\labelenumi}{\rm{(\roman{enumi})}}
%\item
%$
%f(10)=
%10{\downarrow}^G_N\ =\ 
%\boxed{\begin{matrix}   & & k_N& & \\
%                      &\slash&   & \backslash& \\
%                     \ \ \slash & &   &  & 2 \\
%                     |&  &   &  & | \\
%                 1_{kN}&  &   &  & 3_1\\
%                    |  &  &   &  & |  \\
%                    \ \ \backslash&  &   &  & 2^*  \\
%                   & \backslash &   &\slash &   \\
%                      &  & k_N & &   
%                  \end{matrix}}.
%\ \text{ and } \ 
%$
%\item
$\mathfrak f(10')=
10'{\downarrow}^G_N=
U(3_1,2^*,2,3_1)$
%\boxed{\begin{matrix}3_1 \\ 2* \\ 2 \,\ \\ 3_1\end{matrix}}
%=\mathcal H_{Q^\tau}(k_N)
and
$\mathfrak f(10)=
10{\downarrow}^G_N=
U(3_2,2,2^*,3_2) 
%\boxed{\begin{matrix}3_2 \\ 2\ \\ 2^* \\ 3_2\end{matrix}}
=\mathcal H_{Q}(k_N)
$.

%\end{enumerate}
\end{Lemma}
\begin{proof} 
By (\ref{sigma-tau}) it suffices to prove the first part.
The second equality in the first part follows 
by the structure of the $kN$-module $X$ in \cite[the proof of (5.3)Theorem]{KW99b}.
Then, since the $kN$-module at the 3rd term above is indecomposable, 
the first equality follows by the definition of the Green correspondence.
The 3rd equality is obtained by
Lemma~\ref{InductionFromN1}(iv).
\end{proof}

%%%%%\newpage
\begin{Lemma}\label{F(10')} 
%It holds the following;
\item[{}]
\begin{center}
\begin{tikzpicture}
\draw
(-1,-1)node[left](){$F(10)=$}
(0,0)node[above](){$6$}
--(0.5,-0.3)node[below](){$6^*$}
(-0.1,0)--(-0.5,-0.3)node[below](){$3$}
(-0.5,-0.7)--(-0.5,-1)node[below](){$7$}
(-0.5,-1.4)--(-0.5,-1.7)node[below](){$3^*$}
(0.5,-0.7)--(0.5,-1.7)node[below](){$6$}
(0.4,-2.2)--(0.1,-2.5)node[below](){$6^*$}
(-0.5,-2.2)--(0,-2.5)
(-0.3,-0.7)--(0.3,-1.7)
(0.3,-0.6)--(-0.3,-1.7)
;
\end{tikzpicture}
\quad and \quad
\begin{tikzpicture}
\draw
(-1,-1)node[left](){$F(10')=$}
(0,0)node[above](){$6^*$}
--(0.5,-0.3)node[below](){$6$}
(-0.1,0)--(-0.5,-0.3)node[below](){$3^*$}
(-0.5,-0.7)--(-0.5,-1)node[below](){$7$}
(-0.5,-1.4)--(-0.5,-1.7)node[below](){$3$}
(0.5,-0.7)--(0.5,-1.7)node[below](){$6^*$}
(0.4,-2.2)--(0.1,-2.5)node[below](){$6$}
(-0.5,-2.2)--(0,-2.5)
(-0.3,-0.7)--(0.3,-1.7)
(0.3,-0.6)--(-0.3,-1.7)
;
\end{tikzpicture}
\end{center}
\end{Lemma}

\begin{proof}
Enough to prove the first part by (\ref{sigma-tau}). 
First, Lemma~\ref{GreenHeart}(ii) yields that
\begin{equation}\label{Heart}
\mathfrak f'(\mathcal H_Q(k_H))\cong\mathcal H_Q(k_N).
\end{equation}
Then
\begin{align*}
F(10)&=\mathfrak f'^{-1}\circ{\mathfrak f}(10)
\quad\text{ by Lemmas~\ref{vertex}(i) and \ref{stableEqA-B}(iii)}
\\&= \mathfrak f'^{-1}(\mathcal H_{Q}(k_N)) \quad\text{ by Lemma \ref{Res_N(10')}}
\\&= \mathcal H_{Q}(k_H)  \quad\text{\color{black} by (\ref{Heart})}.
%\\&= \mathcal H_{Q}(k_H)) \quad\text{ since }f'\text{ is the Green correspondence},
%
%
%\\&=
%\boxed{
%\begin{matrix}&6^*&\\
%/&&\backslash\\
%3^*&&6\\
%|&\backslash\ /&|\\
%7&/\backslash& |\\
%|&/\ \ \ \backslash&|\\
%3&&6^*\\
%\backslash&&/\\
%&6&
%\end{matrix}
%}
\end{align*}
Hence the assertion holds by Lemma~\ref{H_Q(k_H)}(ii).
\end{proof}

\begin{Lemma}\label{X(10')}
There are complexes of $kH$-modules
\begin{enumerate}
\renewcommand{\labelenumi}{\rm{(\roman{enumi})}}
\item
$$
\begin{matrix}
& & {\small -1\text{st}}\ \ & {\small 0\text{th}}\ \ & 
\\
X^\bullet(10'):\ &0\overset{\partial^{-2}}{\longrightarrow} 
&P(6^*)\overset{\partial^{-1}}{\twoheadrightarrow} &F(10')\overset{\partial^0}{\longrightarrow}&0
\end{matrix}
$$
with cohomology ${\mathrm{H}}^{-1}(X^\bullet(10'))=U(k,7,3^*,6^*)$
%\boxed{
%\begin{matrix}k \ \\ 7\ \\ 3^* \\6^*\end{matrix}
%}$, 
where $\partial^{-1}$ is the canonical epimorphism.
\item
$$
\begin{matrix}
%& & -1\text{st}\ \ & 0\text{th}\ \ & 
%\\
X^\bullet(10):\ &0\overset{\partial^{-2}}{\longrightarrow} 
&P(6)\overset{\partial^{-1}}{\twoheadrightarrow} &F(10)\overset{\partial^0}{\longrightarrow}&0
\end{matrix}
$$
with cohomology ${\mathrm{H}}^{-1}(X^\bullet(10))= 
U(k,7,3,6)$,
%\boxed{
%\begin{matrix}k \\ 7 \\ 3 \\6\end{matrix}}$, 
where $\partial^{-1}$ is the canonical epimorphism.
\end{enumerate}
\end{Lemma}
\begin{proof} (i) follows by Lemmas~\ref{F(10')} and the structure of $P(6^*)$ in
\cite[Theorem~3.8]{Kos86}.
(ii) follows immediately by (i) and (\ref{sigma-tau}).
\end{proof}

%%%%\newpage
\section{The images $F(45)$ and $F(45')$}
\noindent
Recall by Notation~\ref{GreenCorr} that 
$f:=f_{(G,{\color{black}Q},N)}$ and $f':=f_{(H,{\color{black}Q},N)}$ are the Green correspondences 
with respect to $(G,Q,N)$ and $(H,Q,N)$, respectively (see \cite[Chap.4 \S4]{NT88}),
where $N:=N_G(Q)$ and $Q\lneqq P$ is the fixed subgroup with $Q\cong C_3\times C_3$.

\begin{Lemma}\label{Q} %It holds the following:
\begin{enumerate}
\renewcommand{\labelenumi}{\rm{(\roman{enumi})}}
\item 
$Q\in{\mathrm{vx}}({45'}_{kG})$ and $Q^\tau\in{\mathrm{vx}}(45_{kG})$.
%(note that $Q$ and $Q^\tau$ are not $G$-conjugate).
\item
%({\bf\color{red} See Lemma~\ref{like-TI}}).
For any $h\in H-N$ it holds that
{\color{black}$Q\,\cap\,Q^h=1$}.
%so that 
%$Q\,{\not\in}_G\,
%$\mathfrak X:=\mathfrak X(H,Q,N):=\{Q \cap Q^h\,|\, h\in H-N\}$.
%and hence the Green correspondents $f(45')$ and ${f'}^{-1}(f(45'))$ both are defined. 
\end{enumerate}
\end{Lemma}
\begin{proof} (i) follows from \cite[Theorem (ii)]{KW99b}. 

(ii) 
%\footnote{\color{black}So far only by {\sf GAP}. Must check by hand?}
Follows by elementary calculations.
%{\color{black}\bf (A better proof is of course without {\sf GAP}. It should be done
%by \cite[p.13]{Atlas}, I bieve.)}
%
% we have $P\cap P^g=1$
%\text{\bf (Must check by hand as well !!!)}$ 
%for any element $g\in G-N$.
%Thus, $Q{\not\in}\mathfrak X$, and hence $Q{\not\in}_G\mathfrak X$
%by \cite[Chap.4, Lemma 4.1(ii)]{NT88}, 
%that implies that the vertex $Q$ of $45'$ is not in $\mathcal X$, 
%so that the Green correspondent $f(45')$ is defined 
%(see \cite[\S 4]{NT88}). Similar for $f'$. 
\end{proof}

\begin{Notation}\label{Notation3}
Recall that $F(S)$ is an indecomposable $kH$-module for any simple $kG$-module $S$ in $A$
by Lemma~\ref{stableEqA-B} and \cite[Proposition 4.14.9]{Lin18}. %[Theorem 2.1(ii)]{Lin96a}.
Further from now on till the end of this paper we fix the {\it indecomposable} $kH$-module $X$ 
as
$$
{\color{black}X}  :=F(45')={f'}^{-1}\circ f(45')
$$
which plays so important role to complete this paper.
%(cf. Lemma~\ref{Q}(ii)).
Further recall that $\Irr_k(B)=\{k_H, 3, 3^*, 6, 6^*, 7, 15, 15^*\}$
in Notation~\ref{Notation1}.
\end{Notation}

%%%%%\newpage
\begin{Lemma}\label{Res-Ind_45-45'}
%It holds the following:
\begin{enumerate}
\renewcommand{\labelenumi}{\rm{(\roman{enumi})}}
%\item 
%$Q\in{\mathrm{vx}}({45'}_{kG})$ and $Q^\tau\in{\mathrm{vx}}(45_{kG})$,
%(note that $Q$ and $Q^\tau$ is not $G$-conjugate).
%%%%\smallskip
\item
${45'}{\downarrow}^G_N = f(45')\oplus P(3_2)$.
%%%\smallskip
\item
$f(45'){\uparrow}_N^G=45'\oplus{\mathrm{(proj)}}$.
\item
$\Big(f(45'){\uparrow}_N^G\Big){\downarrow}^G_N = f(45')\oplus{\mathrm{(proj)}}$.
\item
{\color{black}
$f(45'){\uparrow}_N^H = X\oplus{\mathrm{(proj)}}$. 
%\bf(By using Lemma~\ref{Q} where {\sf GAP} is used).
}
\item
$(45'{\downarrow}^G_N){\uparrow}_N^H = X\oplus{\mathrm{(proj)}}$.
%%\medskip
\item
$\Omega(F(45'))\cong F(\Omega(45'))$ as $kH$-modules.
\item
$
\dim_k\Big(\underline{\Hom}_{kN}(f(45'), U(2^*, k_N),)\Big)
\ = \ 
[f(45'), U(2^*,k_N)\,]^N = 1.
$
\item
$L_1(X) = 3^*\oplus 6$ and $S^1(X)=3\oplus 6^*$.
\item
$f(45'){\uparrow}_N^H=X\oplus P(15)\oplus P(15*)$.
\item
$X = 2\times k_H+3\times 3_{kH}+3\times {3^*}_{kH}
+2\times 6_{kH}+2\times{6^*}_{kH}+4\times 7_{kH}$ (as composition factors),
and hence $\dim_k\,X=72$.
\item
$(P(3^*)\oplus P(6)){\downarrow}^H_N\cong P(k_N)\oplus P(1_{kN})\oplus P(2^*)\oplus P(3_1)$.
\item
$X{\downarrow}^H_N = f(45')\oplus P(1_{kN})\oplus P(3_1).$
\item
$\Omega^{-1}(6^*){\downarrow}^H_N=(P(6^*)/6^*){\downarrow}^H_N \cong 
\Big(P(3_1)/U(k_N,2,3_1)\Big)
\oplus P(k_N)$.
\end{enumerate}
\end{Lemma}

\begin{proof}
(i) follows from \cite[Theorem(ii)]{KW99b}. 

(ii)
Obviously, 
\begin{equation}\label{f(45')}
f(45'){\uparrow}_N^G=45'\oplus U 
\text{ \ \ for a }kG\text{-module }U,
\end{equation}
which is $\mathcal X(G,{\color{black}Q},N$)-projective.
So we claim that $U$ is projective.
Recall that $L<N$ such that $L\cong Q\rtimes SD_{16}$ by Notation~\ref{Notation1}.
Surely $G$ has the Mathieu group $\sf M_{11}$ of degree $11$ as its subgroup 
with $N_{{\sf M}_{11}}(Q)=L$ by \cite[(7.8)]{KW99b}.
Let $\tilde f:=f_{({\sf M}_{11},Q,L)}$, 
the Green correspondence with respect to $({\sf M}_{11},Q,L)$.
By \cite[p.7]{KW99a}, there is a simple $k\sf M_{11}$-module $24$ of $k$-dimension $24$.
Since $Q\in\Syl_3({\sf M}_{11})$, we have that $24$ cannot have  a cyclic vertex $C_3$ by
\cite[Theorem]{Erd77}, and hence $Q$ is a vertex of $24$.
So that the Green correspondent $\tilde f(24)$ is defined.
%\footnote{Obviously a vertex of $24$ is $Q$ or some $R\leq Q$ with $R\cong C_3$.
%No matter what it is, I mean for both cases, it can be shown that the vertex of $24$
%is in $\mathfrak Z(M_{11}, Q, N_H(Q)=L)$ (see \cite[Chap.4, (4.1) and Theorem 4.3]{NT88})
%since the case that $M_{11}$ and $p=3$ is a {\bf T.I.} situation.} 
Then, 
\begin{align*}
\tilde f(24){\uparrow}_L^G=\, &
\Big(\tilde f(24){\uparrow}_L^{\sf M_{11}}\Big){\uparrow}_{\sf M_{11}}^G
\\ =\,&
\Big( 24\oplus{\mathrm{(proj)}}\Big){\uparrow}_{\sf M_{11}}^G
\text{ since }Q\text{ is a T.I. subset in }{\sf M_{11}}
\text{ (see \cite[(24-1) p.307]{GL83}}) %\cite[Lemma 3.6]{Wak93})}
\\ =\,& 45'\oplus P(99)\oplus
{\mathrm{(proj)}}{\uparrow}_{\sf M_{11}}^G
\text{ by \cite[(3.20)(i)]{KW99a}}
\\ =\,&
45'\oplus{\mathrm{(proj)}},
\end{align*}
so that
\begin{equation}\label{tilde f(24)}
\tilde f(24){\uparrow}_L^G=\, 45'\oplus{\mathrm{(proj)}}.
\end{equation}
Hence,
\begin{align*}
\Big(\tilde f(24){\uparrow}_L^G\Big){\downarrow}^G_N =&\,
(45'\oplus{\mathrm{(proj)}}){\downarrow}^G_N \text{ \quad by (\ref{tilde f(24)}) }
\\
=&\, 45'{\downarrow}^G_N\oplus{\mathrm{(proj)}}
\\
=&\, f(45')\oplus P(3_2)\oplus{\mathrm{(proj)}}\ \ \text{ by (i)}
\\
=&\, f(45')\oplus{\mathrm{(proj)}},
\end{align*}
so that
\begin{equation}\label{tilde f(24)res}
 \Big(\tilde f(24){\uparrow}_L^G\Big){\downarrow}^G_N=f(45')\oplus{\mathrm{(proj)}}.
\end{equation}
Set $U:=\tilde f(24){\uparrow}_L^N$. 
Then, the Mackey decomposition yields that
\[\Big(\tilde f(24){\uparrow}_L^G\Big){\downarrow}^G_N=U\oplus Y
\text{ for a }kN\text{-module }Y.
\]
Hence by (\ref{tilde f(24)res}),
\[
f(45')\oplus{\mathrm{(proj)}}= U\oplus Y.
\]

Suppose that $f(45')\,{|}\,Y$. Then $U$ is projective by the Krull-Schmidt theorem
since $f(45')$ is indecomposable. Hence $U{\downarrow}^N_L$ is projective.
It means that $\tilde f(24)\,|\,\tilde f(24){\uparrow}_L^N{\downarrow}^N_L=U{\downarrow}^N_L$, 
so that $\tilde f(24)$ is projective, which is a contradiction.
 
Thus, $f(45')\,{\not |}\,Y$, so that $f(45')\,|\,U$, and hence $U=f(45')\oplus V$
for a $kN$-module $V$. This implies that
$f(45')\oplus{\mathrm{(proj)}}=f(45')\oplus V\oplus Y$, so that $V$ is projective.
Namely, $U=f(45')\oplus{\mathrm{(proj)}}$. Let $f(45'){\uparrow}_N^G=45'\oplus X$
for a $kG$-module $X$. Then
\begin{equation}\label{final}
\tilde f(24){\uparrow}_L^G=U{\uparrow}_N^G=f(45'){\uparrow}^N_G\oplus{\mathrm{(proj)}}
=45'\oplus X\oplus{\mathrm{(proj)}}.
\end{equation}
Therefore by (\ref{tilde f(24)}) and (\ref{final}),
$45'\oplus{\mathrm{(proj)}}=45'\oplus X\oplus{\mathrm{(proj)}}$, that yields that
$X$ is projective. So (ii) is proved.

%%\medskip

(iii) Easy by (i) and (ii).

(iv) 
This follows by Lemma~\ref{Q}(ii).

(v)
(i) and {\color{black}(iv)} imply that
%\begin{align*}
$(45'{\downarrow}^G_N){\uparrow}_N^H= \Big(f(45')\oplus{\mathrm{(proj)}}\Big){\uparrow}_N^H
= X\oplus{\mathrm{(proj)}}$.
%\end{align*}

(vi) Follows from \cite[Chap.4 Theorem 4.10(ii)]{NT88} and Lemma~\ref{stableEqA-B}(iii)
%%\medskip
%%%%%\newpage

(vii) 
Set $U:=U(2^*,k_N)$.
Take any projective homomorphism $\varphi\in\Hom_{kN}(f(45'), U)$.
Thus, %\cite[Theorem (ii)]{KW99b} and 
by \cite[II Lemma 2.7]{Lan83}, there exist an $\alpha\in\Hom_{kN}(f(45'), P(2^*))$ and
a surjection $\beta\in\Hom_{kN}(P(2^*),U)$ such that
$\varphi=\beta\circ\alpha$.
Since $L_1(f(45'))\cong 2^*$ and $S^1(P(2^*))\cong 2^*$, it holds that the shape of $f(45')$
in \cite[Theorem (ii)]{KW99b} implies that 
${\mathrm{Im}}\,\alpha\cong U(2^*, 2, 2^*)$ or ${\mathrm{Im}}\,\alpha\cong f(45')/U(k_N,2)$,
 so that ${\mathrm{Im}}\,\alpha\subseteq S^6(P(2^*))$ by
the shape of $f(45')$ whichever happens.
On the other hand, $P(2^*)/\Ker\,\beta\cong U$. Thus, the shape of $P(2^*)$ in
\cite[Theorem 3.8]{Kos86} yields that ${\mathrm{Im}}\,\alpha\subseteq\Ker\,\beta$, which
means that $\varphi=0$.
%%\medskip

Throughout the proofs of (viii)--(xii) below we freely use \cite[Theorem~(ii)]{KW99b} 
\cite[Theorem~2.12,7 and Proposition~2.15.4]{Lin18}
%and the Frobenius-Nakayama reciprocity (even for $\underline{\Hom}$) 
without quoting.

(viii)
For any simple $S\in\Irr_k(B)$, we compute
$[X,S]^H$. 

If $S=k_H$, then
$[X,k_H]^H=[{f'}^{-1}(f(45')),k_H]^H\leq [f(45'){\uparrow}_N^H, k_H]^H
=[f(45'), k_H{\downarrow}^H_N]^N = [f(45'), k_N]^N=0$ %by \cite[Theorem (ii)]{KW99b},
and hence 
\begin{equation}\label{[X,k_H]}
[L_1(X),k_H]^H=0.
\end{equation}
If $S=3$, then similarly
$[X,3]^H=[{f'}^{-1}(f(45')),3]^H
\leq [f(45'){\uparrow}_N^H, 3]^H
= [f(45'), 3{\downarrow}^H_N]^N 
=[f(45'), U(2, 1_{kN})\,]^N
=0$
by \cite[(1.3)Proposition (ii)]{Kos87}. 
Namely
\begin{equation}\label{[X,3]}
[L_1(X),3]^H=0.
\end{equation}
Similarly, from \cite[(1.3)Proposition(iii)]{Kos87}, %and \cite[Theorem (ii)]{KW99b}, 
we have
\begin{equation}\label{[X,6^*]}
[L_1(X),6^*]^H=0.
\end{equation}
By the similar way, from
\cite[(1.3)Proposition(iv)]{Kos87} %and \cite[Theorem~(ii)]{KW99b} 
it holds
\begin{equation}\label{[X,7]}
[L_1(X),7]^H=0.
\end{equation}
Furthermore, 
%\begin{comment}
\begin{align*}
\Hom_{kH}(X,3^*)\cong  & \,
\underline{\Hom}_{kH}(X, 3^*) \ \text{ by \cite[Corollary 4.13.4]{Lin18}} 
%\cite[II Corollary 2.8]{Lan83}}
\\ \cong & \,
\underline{\Hom}_{kH}(f(45'){\uparrow}{_N^H}, 3^*) \text{ by 
%Lemma~\ref{Res-Ind_45-45'}
(iv)}
\\ \cong &\,
\underline{\Hom}_{kN}(f(45'), 3^*{\downarrow}^H_N) 
\text{ by \cite[Proposition 2.15.4]{Lin18}}
\\ \cong &\, 
\underline{\Hom}_{kN}(f(45'), U(1,2^*))
\text{ by \cite[(1.3)Proposition (ii)]{Kos87}}
\\ \cong&\,
\underline{\Hom}_{kN}(f(45'), 2^*)%\text{\ by \cite[Theorem(ii)]{KW99b}}
\\ \cong &\,
\Hom_{kN}(f(45'), 2^*) \ \text{ by \cite[Corollary 4.13.4]{Lin18}} 
%\cite[II Corollary 2.8]{Lan83}}
\\ \cong &k \quad \text{ by \cite[Theorem (ii)]{KW99b}}.
\end{align*}
%\end{comment}
Thus %by \cite[Theorem~(ii)]{KW99b}, 
\begin{equation}\label{[X,3^*]}
[L_1(X), 3^*]^H=1.
\end{equation}
Similarly
\begin{align*}
\Hom_{kH}(X, 6) =\, &\underline{\Hom}_{kH}(X, 6) \ \text{ by \cite[Corollary 4.13.4]{Lin18}} 
%\cite[II Corollary 2.8]{Lan83}}
\\
\cong\, &
\underline{\Hom}_{kH}(f(45'){\uparrow}_N^H, 6) 
\ \text{ by (iv)}%Lemma~\ref{Res-Ind_45-45'}(iv)}
\\
\cong\,&
\underline{\Hom}_{kN}(f(45'), 6{\downarrow}^H_N) 
%\\ & \text{ by Nakayama-Frobenius reciprocity}
\\ \cong\, & 
\underline{\Hom}_{kN}(f(45'), U(3_1,2^*,k_N)\,)
\text{\ \ \ from \cite[(1.3)Proposition (iii)]{Kos87}}
\\ 
\cong\, &
\underline{\Hom}_{kN}(f(45'), U(2^*,k_N)\,)
\text{ by \cite[Theorem (ii)]{KW99b}}
\\
\cong\, &
\Hom_{kN}(f(45'), U(2^*,k_N)\,)\cong k
\ \text{ (as }k\text{-spalces), by (vii)} %Lemma~\ref{Res-Ind_45-45'}(vii)}.
%\\ =\,& 1 \ \text{ by \cite[Theorem(ii)]{KW99b}}.
\end{align*}
Thus, 
\begin{equation}\label{[X,6]}
[L_1(X),6]^H=1.
\end{equation}
Finally let us look at the case that $S\in\{15, 15^*\}$. Then
\begin{align*}
\Hom_{kH}(X,15^*) \cong&\, \underline{\Hom}_{kH}(X,15^*) 
\text{\ by \cite[Corollary 4.13.4]{Lin18}}
%\cite[II Corollary 2.8]{Lan83} }
\\ =&\,
\underline{\Hom}_{kH}\Big(F(45'), \,F(F^{-1}(15^*))\Big)\ \text{ from the definition of }F
\\ \cong&\, 
\underline{\Hom}_{kG}\Big(45', F^{-1}(15^*)\Big)\ 
\text{ \ \ from Lemma~\ref{stableEqA-B}(i) }
\\ =&\, \underline{\Hom}_{kG}\Big(45', 15^*\Big)
\text{ \ from Lemma~\ref{F(k)F(15)F(15*)}}
\\ =&\,
0,
\end{align*}
that is
\begin{equation}\label{[X,15^*]} [L_1(X), 15^*]^H=0. \end{equation}
Similarly,
\begin{equation}\label{[X,15]} [L_1(X), 15]^H=0. \end{equation}
Now, since $\Irr_k(B) = \{k_H, 3, 3^*, 6, 6^*, 7, 15, 15^*\}$ and since $X$ is a $kH$-module
in $B$, the first assertion follows from
(\ref{[X,k_H]}) $\sim$
%, (\ref{[X,3]}), (\ref{[X,6^*]}), (\ref{[X,7]}), (\ref{[X,3^*]}),
%(\ref{[X,6]}), (\ref{[X,15^*]}) and 
(\ref{[X,15]}). Hence second one follows since $X$ is self-dual.
%%\medskip

(ix)
Set $Y:=f(45'){\uparrow}_N^H$. Then %Lemma~\ref{Res-Ind_45-45'}
(iv) yields that $Y=X\oplus\mathcal P$
with a projective $kH$-module $\mathcal P$.
Now, it follows by the Frobenius-Nakayama reciprocity, \cite[Theorem(ii)]{KW99b} and
\cite[(1.3)Proposition (i)--(vi)]{Kos87} that 
$L_1(Y)=3^*\oplus 6\oplus 15\oplus 15*$. %and hence 
On the other hand, since $L_1(X)=3^*\oplus 6$ by (i),
$L_1(\mathcal P)=15\oplus 15^*$, so that the assertion follows.

%%\medskip
(x)
%$$ Y=X\oplus P(15)\oplus P(15^*).$$
We keep the notation $Y$ as in the proof of (ii).
Then, (ii) says that in the Grothendieck group,
$X=Y-P(15)-P(15^*)$.
We know the composition factors of $f(45')$ by \cite[Theorem(ii)]{KW99b}, and hence
\cite[(1.3)Proposition (vii)--(xi)]{Kos87} tells us
the composition factors of $Y$. 
Further, the Cartan invariants for $P(15)$ and $P(15^*)$ in \cite[(1.1)Lemma]{Kos87}
give us the information of the composition factors of $P(15)$ and $P(15^*)$.
So, the assertion follows.
%(vii) This follows from Lemma~\ref{Q}(ii) and \cite[Chap.4, Theorem~4.10(ii)]{NT88}.
%%\medskip

(xi)
%It follows from the Frobenius-Nakayama reciprocity and 
By \cite[(1.3)Proposition (vii)--(xi)]{Kos87}, for a simple $kN$-module $\widetilde S$
\begin{align*}
[P(3^*){\downarrow}^H_N, \widetilde S]^N = [P(3*), \widetilde S{\uparrow}_N^H]^H&=
    \begin{cases}
1&\text{ for }\widetilde S\in\{1_{kN}, 2^*\}
\\
0&\text{ otherwise}.
    \end{cases}
\end{align*}
Similarly,
\begin{align*}
[P(6){\downarrow}^H_N, \widetilde S]^N = [P(6), \widetilde S{\uparrow}_N^H]^H&=
   \begin{cases}
1&\text{ for }\widetilde S\in\{k_N, 3_1\}
\\
0&\text{ otherwise}.
   \end{cases}
\end{align*}
Thus, the assertion follows.

%%\medskip

(xii)
By (viii), the projective cover $P(3^*)\oplus P(6)=P(X)\twoheadrightarrow X$ exists.
Apply this to the restriction ${\downarrow}^H_N$. By the definitions of the Green correspondence
$f'$ and $X$, we can write $X{\downarrow}^H_N=f(45')\oplus U$ for a $kN$-module $U$.
Thus, (xi) yields that a $kN$-epimorphism
\begin{equation}\label{ResX}
P(k_N)\oplus P(1_{kN})\oplus P(2^*)\oplus P(3_1)\twoheadrightarrow X{\downarrow}^H_N=f(45')\oplus U.
\end{equation}
%Then, since $P(f(45'))\cong P(2^*)$ by \cite[Theorem (ii)]{KW99b}, we do have a 
%$kN$-epimorphism
%\begin{equation}\label{U}
%P(k_N)\oplus P(1_{kN})\oplus P(3_1)\twoheadrightarrow U.
%\end{equation}
Now,  for $\tilde S\in\Irr_k(kN)$, \cite[(1.3)Proposition (vii)--(xi)]{Kos87} implies that
\begin{align*}
[X{\downarrow}^H_N,\tilde S]^N&= [X,\tilde S{\uparrow}^H]^H
\\
&\geq [3^*\oplus 6,\tilde S{\uparrow}^H]^H \text{ \ from (viii)}
\\
&=
   \begin{cases}
1 &\text{ if }\tilde S=1_{kN}, 3_1
\\
0 &\text{ otherwise}.
   \end{cases}
\end{align*}
Then, since $L_1(f(45'))=2^*$ by \cite[Theorem (ii)]{KW99b}, (\ref{ResX}) yields that
\begin{equation}\label{U1}
[U,1_{kN}]^N\geq 1\text{ and }[U,3_1]^N\geq 1.
\end{equation}
On the other hand, by Lemma~\ref{stableEqA-B}(iii), $X\,|\,f(45'){\uparrow}^H_N$,
so that (iii) yields that
$$X{\downarrow}^H_N\ |\ ( f(45'){\uparrow}^H_N){\downarrow}^H_N=f(45')\oplus{\mathrm{(proj)}}.$$
Hence, by (\ref{ResX}), $U$ has to be projective. Thus, $P(1_{kN})\oplus P(3_1)\,|\,U$ by (\ref{U1}).
Then by comparing the $k$-dimensions of the both sides by \cite[Theorem 3.8]{Kos86},
\cite[Theorem (ii)]{KW99b} and (x), we do have.
$U=P(1_{kN})\oplus P(3_1)$. The proof is completed.

%%\medskip

(xiii)
Recall that $\Omega^{-1}(6^*) = I(6^*)/6^*=P(6^*)/6^*$.
Then, 
\begin{align*}
 \Big(\Omega^{-1}(6^*)\Big){\downarrow}^H_N
&= \Omega^{-1}(6^*{\downarrow}^H_N)\oplus{\mathrm{(inj)}}
\\
&= \Omega^{-1}(6^*{\downarrow}^H_N)\oplus{\mathrm{(proj)}}
\\
&= (P(3_1)/U(k_N,2,3_1))%\boxed{\begin{matrix}k_N\\ 2\ \\ 3_1\end{matrix}}\ \Big)
\oplus\text{(proj)}
\text{ \ \ \ from \cite[(1.3)Proposition (iii)]{Kos87}}.
\end{align*}
Further, %it follows from the Frobenius-Nakayama reciprocity and 
for a simple $kN$-module $\widetilde S$,
by \cite[(1.3)Proposition (vii)--(xi)]{Kos87} 
\begin{align*}
[(P(6^*)/6^*){\downarrow}^H_N, \widetilde S]^N = [P(6^*)/6^*, \widetilde S{\uparrow}_N^H]^H&=
\begin{cases}
1&\text{ for }\widetilde S\in\{ k_N, S_1\}
\\
0&\text{ otherwise}.
\end{cases}
\end{align*}
Thus, all the assertions follow.
\end{proof}

\begin{Lemma}\label{mu'} %It holds the following:
\begin{enumerate}
\renewcommand{\labelenumi}{\rm{(\roman{enumi})}}
\item
There is a $kH$-module homomorphism 
$$
\mu': P(6^*) \rightarrow \ P(3^*)\oplus P(6) \text{ \ such that }
{\mathrm{Ker}}(\mu') = 6^*.
$$
(We shall use $\mu'$ in this sense below.)
\item
Set $\iota$ be an $kH$-monomorphism $\iota: P(6^*)/6^* \rightarrowtail P(3^*)\oplus P(6)$
induced by $\mu'$, that is,
$\mu'=\iota\circ\pi$ where $\pi$ is the canonical epimorphism 
$\pi:P(6^*){\twoheadrightarrow}P(6^*)/6^*$.
Further, let $\pi_X$ be the $kH$-epimorphism defined by the projective cover of $X$, namely,
$P(3^*)\oplus P(6) = P(X) \overset{\pi_X}{\twoheadrightarrow} X$,
and hence there exists a sequence of $kH$-modules
$$
X^\bullet(45'): \ \ 
\Big[P(6^*)\ \overset{\mu'}{\longrightarrow} \ P(3^*)\oplus P(6)\ 
\overset{\pi_X}{\twoheadrightarrow}X\Big].
$$
Furthermore, there is a sequence of $kN$-modules
\begin{align*}
\Big(X^\bullet(45')\Big){\downarrow}^H_N
= \ 
\Big[P(3_1)\oplus P(k_N) 
&\overset{\mu'{\downarrow}_N}{\longrightarrow} P(2^*)\oplus P(k_N)\oplus P(1_{kN})\oplus P(3_1)
\\
&\overset{\pi_X{\downarrow}_N}{\twoheadrightarrow} f(45')\oplus P(1_{kN})\oplus P(3_1)\Big].
\end{align*}
(We shall use $\iota$ and $\pi_X$ in this sense below.)
\item
We may assume that
$$ 
{\mu'{\downarrow}_N}(P(3_1))\subseteq P(2^*)\text{ \ \ and \ \  }
{\pi_X{\downarrow}_N}(P(2^*))\subseteq f(45')
$$
namely, we can consider that there is a sequence of $kN$-modules
\begin{align*}
\mathfrak Y^\bullet :
&=\Big[P(3_1) \overset{\mu'{\downarrow}_N}{\longrightarrow} P(2^*)
\overset
{\pi_X{\downarrow}_N}{\twoheadrightarrow} f(45')\Big]
\\
&\text{ with } {\mu'}{\downarrow}_N= 
\Big(L\overset{\iota{\downarrow}_N}{\subseteq} P(2^*) \Big)
\circ
\Big(P(3_1)\overset{\pi}{\twoheadrightarrow}L\Big)
%\boxed{\begin{matrix}k_N\\ 2\ \\ 3_1\end{matrix}} 
\\
&\text{ where }L:=P(3_1)/U(k_N,2,3_1)\text{ and }\pi\text{ is the canonical onto map.}
%\overset{\iota{\downarrow}_N}{\subseteq} P(2^*).
\end{align*}
%We can consider that $\mu'{\downarrow}_N\Big(P(3_1)\Big)\subseteq P(2^*)$ in 
%$\Big(\mathfrak X(45')\Big){\downarrow}^H_N$ (cf. {\rm{(iii)}}), and hence
%we can consider that 
%$$
%\Big(\mathfrak X(45')\Big){\downarrow}^H_N
%\ = \ 
%\Big[P(3_1) \overset{\mu{\downarrow}_N}{\longrightarrow} P(2^*)
%                       \overset{\pi_X{\downarrow}_N}{\longrightarrow} f(45')\Big].
%$$
\item
$\pi_X\circ\mu'=0$.
%%\bigskip
\item
${\mathrm{Im}}\,\iota={\mathrm{Im}}\,\mu' \subseteq {\mathrm{Ker}}(\pi_X)$,
that is,
$X^\bullet(45')$ in (ii) is a complex of $kH$-modules in $B$.
\item
$\Big(P(3^*)\oplus P(6)\Big)/{\mathrm{Im}}\,\mu' 
= \Big(P(3^*)\oplus P(6)\Big)/\Big(P(6^*)/6^*\Big)\cong 15^*$.
\item
$$
\begin{tikzpicture}
\draw 
(-1.3,-1.8)node[left](){$\Omega F(45)=$}
(0,0) node[above](){$6$}--(0.3,-0.3)node[below](3){\quad $3$}
(0.6,-0.7)--(0.9,-1)node[below](7){$7$}
(-0.1,0)--(-1.1,-1)node[below](){\ $6^*$}
(-1.1,-1.4)--(-1.1,-1.7)node[below](){$6$}
(0.4,-0.7)--(-0.9,-2)
(0.9,-1.4)--(0.9,-1.8)node[below](){$k$}
(0.7,-1.3)--(0.2,-1.8)node[below](){$3^*$  }
(-1.1,-2.1)--(-1.1,-3.3)
    %(-0.9,-2.6)--(-0.1,-2)
(0.9,-2.1)--(0.9,-2.4)node[below]{$7$}
     %(-1.1,-2.8)--(-0.1,-3.9)node[below]{$6$}
(0.9,-2.9)--(0.9,-3.3)node[below]{$3$}
     %(0.1,-4)--(0.3,-3.7)
     %(-0.9,-2.1)--(0.2,-3.3)
(-1,-1.3)--(-0.1,-1.9)
(0.3,-2.1)--(0.8,-2.5)
(0,-2.1)--(-1,-3.3)node[below](){\ $6^*$ \, }
(2,0)node[above](){$15$}--(1.1,-1.9)
(-0.9,-2.1)--(0.7, -3.5)
(3.1,-1.8)node[left](){and}
;
\end{tikzpicture}
\begin{tikzpicture}
\draw 
(-1.7,-1.8)node[left](){$\Omega F(45')=$}
(0,0) node[above](){$6^*$}--(0.3,-0.3)node[below](3){\quad $3^*$}
(0.6,-0.7)--(0.9,-1)node[below](7){$7$}
(-0.1,0)--(-1.1,-1)node[below](){\ $6$}
(-1.1,-1.4)--(-1.1,-1.7)node[below](){$6^*$}
(0.4,-0.7)--(-0.9,-2)
(0.9,-1.4)--(0.9,-1.8)node[below](){$k$}
(0.7,-1.3)--(0.2,-1.8)node[below](){$3$  }
(-1.1,-2.1)--(-1.1,-3.3)
    %(-0.9,-2.6)--(-0.1,-2)
(0.9,-2.1)--(0.9,-2.4)node[below]{$7$}
     %(-1.1,-2.8)--(-0.1,-3.9)node[below]{$6$}
(0.9,-2.9)--(0.9,-3.3)node[below]{$3^*$}
     %(0.1,-4)--(0.3,-3.7)
     %(-0.9,-2.1)--(0.2,-3.3)
(-1,-1.3)--(-0.1,-1.9)
(0.3,-2.1)--(0.8,-2.5)
(0,-2.1)--(-1,-3.3)node[below](){\ $6$ \, }
(2,0)node[above](){$15^*$}--(1.1,-1.9)
(-0.9,-2.1)--(0.7, -3.5)
;
\end{tikzpicture}
$$
\end{enumerate}
\end{Lemma}
\begin{proof}
(i) Follows by \cite[Theorem 1]{Kos87}.

(ii) Follows by Parts (viii), (xi), (xii) and (xiii) of Lemma~\ref{Res-Ind_45-45'}.

%\medskip

(iii)
First, it follows from \cite[(1.3)Proposition]{Kos87} and \cite[Theorem~3.8]{Kos86} that
$$ {\mathrm{Im}}(\mu'{\downarrow}_N)\subseteq P(2^*)\oplus P(k_N)$$
though $(P(3^*)\oplus P(6)){\downarrow}_N=(P(1_{kN})\oplus P(2^*))\oplus (P(3_1)\oplus P(k_N))$.
Further, recall that 
$X{\downarrow}_N=f(45')\oplus P(1_{kN})\oplus P(3_1)$ by (iii).
Hence, Lemma~\ref{ProjCover}(ii) and the definitions of projective covers and injective hulls
imply that we can exclude from $\Big(X^\bullet(45')\Big){\downarrow}^H_N$
the projective modules $P(k_N)$\,s in the 1st and 2nd terms,
and also $P(1_{kN})$\,s, $P(3_1)$\,s in the 2nd and the 3rd terms.
Hence we get the desired sequence $\mathfrak Y^\bullet$.

%\medskip

%%%%\newpage
(iv)
%By the original definition of $\mu'$ in (i), 
%$$P(6*)/6^* \subseteq P(3^*)\oplus P(6),$$
%and hence 
%$$L:=P(3_1)\,\Big/\,\boxed{\begin{matrix}k_N\\ 2\ \\ 3_1\end{matrix}} \ \ \subseteq \ P(2^*).$$
%Now, 
Assume that $\pi_X\circ\mu'\,{\not=}\,0$. Then of course
$\pi_X{\downarrow}_N\circ\mu'{\downarrow}_N\,{\not=}\,0$. 
Then, since $P(f(45'))=P(2^*)$ by \cite[Theorem (ii)]{KW99b},
it follow from Lemma~\ref{ProjCover} that $(\pi_X{\downarrow}_N)(P(k_N))=0$.
Hence again it follows from Lemma~\ref{ProjCover}, (ii) and (iii) that we can consider that
\linebreak
$(\pi_X{\downarrow}_N)\circ(\mu'{\downarrow}_N)(P(3_1))\,{\not=}\,0.$ 
Namely,
$$
 0 \,{\not=}\,(\pi_X{\downarrow}_N\circ\mu'{\downarrow}_N)(P(3_1))
=(\pi_X{\downarrow}_N\circ\iota)(L)=\pi_X\Big(\iota(L)\Big)=:\mathfrak L
$$
where $L$ is the same as in (iii).
Since $\iota$ is injective and since $L_1(L)=3_1$, and since $\pi_X$ is surjective,
$c_{\mathfrak L}(3_1)\,{\not=}\,0$. 
Then, since $\mathfrak L\subseteq f(45')$, $c_{f(45')}(3_1)\,{\not=}\,0$,
contradicting \cite[Theorem (ii)]{KW99b}.
Therefore, the assertion follows.

%\medskip

(v) Obvious by (iv).

(vi) Follows from the definition of $\mu'$.

(vii)
By (iv), ${\mathrm{Im}}\,\mu'\subseteq\Omega X$. Hence (vi) implies that
$\Omega X/\Big(P(6^*)/6^*\Big)\cong 15^*$. Thus, the pictures of $P(15^*)$ and $P(6^*)$
in Lemma~\ref{P(6)-P(15)} yield the assertion.
\end{proof}

\begin{Notation}\label{Notation5}
From now on till the end of this paper we shall use the notation $\mu'$, $\pi_X$, $\iota$
and $X^\bullet(45')$
as in Lemma~\ref{mu'}.
\end{Notation}

%%%%\newpage
\section{Simple objects in ${\mathrm{D}}^b({\mathrm{mod}}\text{-}B)$}

\begin{comment}
%%\bigskip
\begin{table}[h]
\caption{Example}
\label{table1}
\begin{center}
\begin{tabular}{c | cccc}
$*$ & $1$ & $3$ & $5$ & $7$\\
\hline
1 & 1 & 3 & 5 & 7 
\end{tabular}
\end{center}
\end{table}
%%\bigskip
\end{comment}

\begin{Notation}\label{Notation4}
Throughout this section we use the following notation: Set
\\
$\mathcal C:={\mathrm{C}}^{\mathrm{b}}(\text{mod-}B)$,
$\mathcal D:={\mathrm{D}}^{\mathrm{b}}(\text{mod-}B)$,
$\mathcal K:={\mathrm{K}}^{\mathrm{b}}(\text{mod-}B)$,
%$\mathcal K_{\mathrm{acy}}:={\mathrm{K}}^{\mathrm{b}}(\text{mod-}B)_{\mathrm{acy}}$
%that is a faithful full subcategory of $\mathcal K$ whose objects are
%all objects of $\mathcal K$ which are {\sf acyclic},
%and hence 
%\begin{equation}\label{DerivedAcyclic}
%\mathcal D\cong\mathcal K/\mathcal K_{\mathrm{acy}}
%\end{equation} 
%as well known
%since $\text{mod-}B$ is an idempotent and exact category
%see \cite[1.3 p.280]{Ric01}). 
%%and \cite[Definition 7.1.15]{Nak15}).
%Further set
$\mathfrak K:={\mathrm{K}}^{\mathrm{b}}(\text{proj-}B)$.
\\
%, then recall Rickard's theorem
%(see \cite[Theorem 2.1]{Ric89})
%\begin{equation}\label{Ric89Stable}
%\underline{\text{mod}}\text{-}B\cong\mathcal D/\mathfrak K\ .
%\end{equation}
%
We use the notation $X^\bullet(S)$ for $S\in\Irr_k(A)$ which are defined in 
Lemma~\ref{complex} below. 
\end{Notation}

\begin{Lemma}\label{complex} %It holds the following:
%\renewcommand{\labelenumi}{\rm{(\roman{enumi})}}
%\begin{enumerate}
%\item
The functor $F:{\mathrm{mod}}\text{-}A\rightarrow{\mathrm{mod}}\text{-}B$
produces the following complexes.
\begin{table}[h]
\vspace{-1cm}
\caption*{}
%\label{table2}
\centering
 \begin{tabular}{l |ccccccccc| l}
{\rm degree} &$-3$& &$-2$ && $-1$ && $0$ &&$1$& ${\mathrm{H}}^i:={\mathrm{H}}^i(X^\bullet(S))$ \\
 \hline
&&&&&&&&& \\
$X^\bullet(k_G)$ & $0$&$\rightarrow$& $0$ &$\rightarrow$& $0$ &$\rightarrow$& $F(k_G)$ &
     $\rightarrow$&$0$& ${\mathrm{H}}^0=k_H$ \\ 
$X^\bullet(34)$ & $0$&$\rightarrow$&$0$ &$\rightarrow$& $0$ &$\rightarrow$& $F(34)$ &
     $\rightarrow$&$0$& ${\mathrm{H}}^0=7$ \\ 
$X^\bullet(15)$ & $0$&$\rightarrow$&$0$ &$\rightarrow$& $0$ &$\rightarrow$& $F(15)$ &
     $\rightarrow$&$0$& ${\mathrm{H}}^0=15$ \\   
$X^\bullet(15^*)$ & $0$&$\rightarrow$&$0$ &$\rightarrow$& $0$ &$\rightarrow$& $F(15^*)$ &
     $\rightarrow$&$0$& ${\mathrm{H}}^0=15^*$ \\     
$X^\bullet({10})$ & $0$&$\rightarrow$&$0$ &$\rightarrow$& $P(6)$ 
&$\twoheadrightarrow$& $F({10})$ &
     $\rightarrow$&$0$& ${\mathrm{H}}^{-1}=U(k_H,7,3,6)$ \\  
$X^\bullet(10')$ & $0$&$\rightarrow$&$0$ &$\rightarrow$& $P(6^*)$ 
&$\twoheadrightarrow$& $F(10')$ &
     $\rightarrow$&$0$& ${\mathrm{H}}^{-1}=U(k_H,7,3^*,6^*)$ \\   
$X^\bullet(45)$ & $0$&$\rightarrow$&$P(6)$ &$\rightarrow$& $P(3)\oplus P(6^*)$ 
&$\twoheadrightarrow$& $F(45)$ &
     $\rightarrow$&$0$& ${\mathrm{H}}^{-2}=6, \ \  {\mathrm{H}}^{-1}=15$ \\   
$X^\bullet(45')$ & $0$&
$\rightarrow$
&$P(6^*)$ &$\overset{\mu'}{\rightarrow}$& $P(3^*)\oplus P(6)$ 
&$\twoheadrightarrow$& $F(45')$ &
     $\rightarrow$&$0$& ${\mathrm{H}}^{-2}=6^*, \ {\mathrm{H}}^{-1}=15^*$   
\end{tabular}
%\vspace{0.3cm}
% \caption{}%The complexes}
 \label{complexes}
\end{table}
\end{Lemma}
\begin{proof}
Follows from Lemmas~\ref{F(k)F(15)F(15*)}, \ref{F(34)} and \ref{mu'}(iv), \ref{mu'}(ii) 
and %also by making use of $\tau\in\Out(H)$ in 
(\ref{sigma-tau}).
\end{proof}

\begin{Lemma}\label{F(S)simple}
For $S, S'\in\{k_G,34,15,15^*\}$,
$\dim_k\,\Hom_{\mathcal D}(X^\bullet(S), X^\bullet(S'))=\delta_{S,S'}$.
\end{Lemma}
\begin{proof}
From Lemmas~\ref{F(k)F(15)F(15*)} and \ref{F(34)},
$F(S)$ and $F(S')$ are simple $B$-modules. Hence,
\begin{align*}
&\Hom_B(F(S),F(S'))\cong\underline{\Hom}_B(F(S),F(S'))
\text{ by \cite[II Corollary 2.8]{Lan83})}
\\
\cong&\,\underline{\Hom}_A(S,S') \text{ by Lemma~\ref{stableEqA-B}(i)}
\\
\cong&\,\Hom_A(S,S') \text{ by \cite[II Corollary 2.8]{Lan83})}
\\
\cong&\,\delta_{S,S'}\times k \ (\text{as }k\text{-spaces})
\text{ by Schuer's Lemma}.
\end{align*}
\end{proof}

\begin{Lemma}\label{F(k)vsF(45)}
For $S\in\{k_G,34,15,15^*\}$ and $S'\in\{45,45'\}$
\[
\Hom_{\mathcal D}(X^\bullet(S),X^\bullet(S'))=0
\text{ and }\, \Hom_{\mathcal D}(X^\bullet(S'),X^\bullet(S)).
\]
\end{Lemma}
\begin{proof}
Assume first that %{\sf Case 1.} 
$S:=k_G$ and $S':=45'$.
Take any 
%${\mathrm{class}}(f^\bullet)\in$
%$\Hom_{\mathcal K}(X^\bullet(k_G),X^\bullet(45'))$with 
$f^\bullet\in\Hom_{\mathcal C}(X^\bullet(k_G),X^\bullet(45'))$.
Then we have by Lemma~\ref{complex} the diagram
\[
\begin{tikzcd}
X^\bullet(k_G)\arrow[d,"f^\bullet"]:  0\arrow[r] &0\arrow[r]\arrow[d] &0 \arrow[d]\arrow[r,"\partial^{-1}"] 
                   & k_H\arrow[d,"f^0"]\arrow[r,"\partial^0"] &0\\
X^\bullet(45'):
0 \arrow[r]& P(6^*)\arrow[r]&\arrow[r,"\delta^{-1}"]P(3^*)\oplus P(6) &\arrow[r,"\delta^0"] F(45') &0
\end{tikzcd}
\]
Since $k_H\,{\not |}\, S^1(F(45'))$ by Lemma~\ref{Res-Ind_45-45'}(viii),
$f^0=0$, namely $f^\bullet=0$, which yields that
%\linebreak
$\Hom_{\mathcal C}(X^\bullet(k_G),X^\bullet(45'))=0$.
%Thus $\Hom_{\mathcal D}(X^\bullet(k_G),X^\bullet(45'))=0$.
The rest is proved similarly.
\end{proof}

\begin{Lemma}\label{Hom(F(10),F(10))}
For $S, S'\in\{10,10'\}$,
$\dim_k\,\Hom_{\mathcal D}(X^\bullet(S), X^\bullet(S'))=\delta_{S,S'}$.
\end{Lemma}
\begin{proof}
First, consider the case that $S=S'=10$. Set $X^\bullet:=Y^\bullet:=X^\bullet(10)[-1]$.
Then, $\H^0(X^\bullet)=U(k_H,7,3,6)=:U$ and $\H^i(X^\bullet)=0$ for any $i\,{\not=}\,0$ by 
Lemma~\ref{complex}.
Thus 
\linebreak
$\Ext_B^j(\H^1(X^\bullet),\H^0(Y^\bullet))=0$ for $j\in\{1,2\}$ since $\H^1(X^\bullet)=0$.
Clearly, $\Hom_B(\H^0(X^\bullet), \H^0(X^\bullet))=\Hom_B(U,U)\cong k$ (as $k$-spaces)
and $\Hom_B(\H^1(X^\bullet), \H^1(X^\bullet))=0$.
Hence we can apply {\color{black}Lemma~\ref{Zimmer-6.6.2}}. Thus,
the homomorphisms $f$ and $g$ in there are, in fact, just
$k\overset{f}{\rightarrow} 0$ and $0\overset{g}{\rightarrow}0$. 
So that the pull-back showing up 
their is $k$ (1-dimensional $k$-space), which means that
the assertion holds for $S=S'=10$.

Next, consider the case that $S=10$ and $S'=10'$. Set $X^\bullet$ as the same as above,
and $Y^\bullet:=X^\bullet(10')[-1]$. Recall that $\H^0(Y^\bullet)=U(k_H,7,3^*,6^*)$ by Lemma~\ref{complex}.
Hence $\Hom_B(\H^0(X^\bullet),\H^0(Y^\bullet))=0$. Thus, just as above,
{\color{black}Lemma~\ref{Zimmer-6.6.2}} implies the assertion.

For the other two cases such that $(S,S')\in\{(10',10'),(10',10)\}$,
the assertion follows from what we have done above and (\ref{sigma-tau}).
\end{proof}

\begin{Lemma}\label{Hom(X(45),X(10))}
For $S\in\{45,45'\}$ and $S'\in\{10,10'\}$,
$$
\Hom_{\mathcal D}(X^\bullet(S),X^\bullet(S')) = 0\text{ and }
\Hom_{\mathcal D}(X^\bullet(S'),X^\bullet(S)) = 0.
$$
\end{Lemma}
\begin{proof}
We first look at the case that $S:=45$ and $S':=10'$.
Consider the first equality.
Set $X^\bullet:=X^\bullet(45)[-2]$ and $Y^\bullet:=X^\bullet(10')[-2]$. 
Then Lemma~\ref{complex} implies
that
$$
\H^i(X^\bullet)=
\begin{cases}
\ 6&\text{ for }i=0\\
15&\text{ for }i=1\\
\ 0&\text{ for }i\,{\not\in}\,\{0,1\}
\end{cases}
\text{ \quad and \quad}
\H^i(Y^\bullet)=
\begin{cases}
0&\text{ for }i=0\\
U(k_H,7,3^*,6^*)&\text{ for }i=1\\
0&\text{ for }i\,{\not\in}\,\{0,1\}.
\end{cases}
$$
Set $U:=U(k_H,7,3^*,6^*)$.
Thus,
$\Ext_B^1(\H^1(X^\bullet),\H^0(Y^\bullet))=\Ext_B^1(15, 0)=0$,
and hence we can apply {\color{black}Lemma~\ref{Zimmer-6.6.2}}.
%\linebreak
Obviously,
$\Hom_B(\H^0(X^\bullet),\H^0(Y^\bullet))=\Hom_B(6,0)=0$ and
$\Hom_B(\H^1(X^\bullet),\H^1(Y^\bullet))=\Hom_B(15,U)=0$.
So the first equality follows by {\color{black}Lemma~\ref{Zimmer-6.6.2}}. 

Next, let us look at the second equality. 
If $\Ext_B^1(U,6)\,{\not=}\,0$, then there exists a $B$-module
$V:=U(k_H,7,3^*,6^*,6)$ since $S^2(P(6))=3\oplus 6^*$ by Lemma~\ref{P(6)-P(15)}(ii),
which is a contradiction since $V\subseteq P(6)$ (see Lemma~\ref{P(6)-P(15)}(ii)).
%the fact that $6\,{\not|}\,L_5(P(k_H))$ by \cite[Theorem 1]{Kos87}.
Thus, $0=\Ext_B^1(U,6)=\Ext_B^1(\H^1(Y^\bullet),\H^0(X^\bullet))$.
Hence we can apply {\color{black}Lemma~\ref{Zimmer-6.6.2}}.
Obviously, $\Hom_B(\H^0(Y^\bullet),\H^0(X^\bullet))=\Hom_B(0,6)=0$ and
$\Hom_B(\H^1(Y^\bullet),\H^1(X^\bullet))=\Hom_B(U,15)=0$.
Thus, just as in the above proof, we get the second equality from 
{\color{black}Lemma~\ref{Zimmer-6.6.2}}.

For the other cases $(S,S')\in\{(45,10),(45',10'),(45,10)\}$, the two equations follow
by similar method  and (\ref{sigma-tau}).

\end{proof}

%%%%%\newpage
\begin{Lemma}\label{F(k)vsF(10)}
For $S\in\{k_G,34,15,15^*\}$ and $S'\in\{10,10'\}$
\[
\Hom_{\mathcal D}(X^\bullet(S),X^\bullet(S'))=0\text{ \ \ and \ \ }
\Hom_{\mathcal D}(X^\bullet(S'),X^\bullet(S))=0.
\]
\end{Lemma}
\begin{proof}
Similar to the proof of Lemma~\ref{Hom(X(45),X(10))}.
\end{proof}

\begin{Lemma}\label{Ext^2}
$\dim_k\,\Ext_B^2(15,6)=1$.
\end{Lemma}
\begin{proof}
By Lemma~\ref{P(6)-P(15)},
$\Hom_B(\Omega 15,\Omega^{-1}6)\cong\End_B(U(k_H,7,3))\cong k$, as $k$-spaces,
and hence
\begin{equation}\label{End} 
\Hom_B(\Omega 15,\Omega^{-1}6)\cong k\text{ as }k\text{-spaces}.
\end{equation}
Secondly, take any $\varphi\in\Hom_B(\Omega 15,\Omega^{-1}6)$ which is relatively-projective.
Then, by \cite[II Lemma 2.7]{Lan83}, there exists $\psi\in\Hom_B(I(\Omega 15),\Omega^{-1}6)$
such that $\varphi=\psi\circ\tau$ where $\tau$ is the monomorphism in 
$\Hom_B(\Omega 15,I(\Omega 15))$ given by the injective hull. 
So that $\psi\in\Hom_B(P(15),\Omega^{-1}6)$ since $I(\Omega 15))=P(15)$.
Assume that $\varphi\,{\not=}\,0$. So $\psi\,{\not=}\,0$, and hence
$15\,|\,L_1(\Im\,\psi)$. However, $\Im\,\psi\subseteq\Omega^{-1}6$ and $\Omega^{-1}6$ is a
quotient of $P(6)$, which implies that $c_{15,6}\,{\not=}\,0$, a contradiction by 
Lemma~\ref{P(6)-P(15)}. Thus, 
\begin{equation}\label{Omega}
{\underline{\Hom}}_B(\Omega 15,\Omega^{-1}6)\cong\Hom_B(\Omega 15,\Omega^{-1}6)
\text{ \ as }k\text{-spaces}.
\end{equation}
Therefore,
\begin{align*}
\Ext_B^2(15,6)&\cong\Ext_B^1(15,\Omega^{-1}6)%\text{ \ by \cite[Proposition 2.5.7(i)]{Ben98-I} }
\\
&\cong {\underline{\Hom}}_B(\Omega 15,\Omega^{-1}6)\text{ \ by \cite[II Corollary 2.6]{Lan83} }
\\
&\cong\Hom_B(\Omega 15,\Omega^{-1}6)\text{ \ by (\ref{Omega})  }
\\
&\cong k \text{ \ by (\ref{End}). }
\end{align*}
\end{proof}

\begin{Lemma}\label{Hom(X(45),X(45))}
For $S, S'\in\{45,45'\}$,
$
\dim_k\,
\Hom_{\mathcal D}(X^\bullet(S),X^\bullet(S'))=\delta_{S,S'}.
$
\end{Lemma}
\begin{proof}
The method to prove this lemma is essentially the same as in
Lemmas~\ref{Hom(F(10),F(10))} and \ref{Hom(X(45),X(10))}, or even easier 
because the cohomology groups $\H^i(X^\bullet)$ and $\H^i(Y^\bullet)$ showing up
are just simple modules unless zero modules, once we can use Lemma~\ref{Ext^2}
and we notice the following. Namely,
when we want to apply {\color{black}Lemma~\ref{Zimmer-6.6.2}} to the case that 
e.g. $S=S'=45$, the pull-back showing up there is just $(k\oplus k)/\Delta$ where
$\Delta:=\{(\alpha,\alpha)\in k\oplus k|\alpha\in k\}$, and this is isomorphic to
$\Hom_{\mathcal D}(X^\bullet(45),X^\bullet(45))$, so that its $k$-dimension is one.
Of course, $id_{X^\bullet(45)}\,{\not=}\,0$ in $\Hom_{\mathcal D}(X^\bullet(45),X^\bullet(45))$.
\end{proof}

%\newpage
%{\color{blue}
\begin{Definition}\label{complex-Y}
From now on till the end of this paper, set
\begin{align*}
Y^\bullet &:=[0\rightarrow P(6)\overset{\partial^{-2}}{\rightarrow}\Omega F(45)\rightarrow 0]
\ \ \in{\mathrm{Ob}}(\mathcal C)
\text{ where }\partial^{-2}: P(6)\twoheadrightarrow P(6)/6 \subseteq \Omega F(45)
\\
(Y')^\bullet &:=[0\rightarrow P(6^*)\overset{\partial^{-2}}{\rightarrow}\Omega F(45')\rightarrow 0]
\ \ \in{\mathrm{Ob}}(\mathcal C)
\text{ where }\partial^{-2}: P(6^*)\twoheadrightarrow P(6^*)/6^* \subseteq \Omega F(45')
\end{align*}
(note that it makes sense by Lemma~\ref{mu'}(vii) ).
\end{Definition}

\begin{Lemma}\label{45}
$X^\bullet(45)\cong Y^\bullet\text{ and }X^\bullet(45')\cong (Y')^\bullet$ in 
$\mathcal D$.
\end{Lemma}
\begin{proof}
By Lemma~\ref{complex},
\begin{equation*}
X^\bullet :=X^\bullet(45'): \quad
0\rightarrow P(6^*)\overset{d^{-2}}{\rightarrow}P(3^*)\oplus P(6)
\overset{d^{-1}}{\rightarrow}F(45')\rightarrow 0.
\end{equation*}
Then, it follows from Lemma~\ref{mu'}(iv) that
$0=d^{-1}\circ d^{-2}(P(6^*))=d^{-1}(P(6^*)/6^*)$ and
$d^{-1}$ is the epimorphism defined by the projective cover of $F(45')$, and hence
$P(6^*)/6^*\subseteq\Ker(d^{-1})=\Omega(F(45'))$.  
Namely, $\Im(d^{-2})\subseteq\Omega F(45')\subseteq P(3^*)\oplus P(6)$.
Thus, by setting that $f^{-2}:=id_{P(6^*)}$ and that $f^{-1}$ is just the inclusion map
as below, we get the following
\[
\begin{tikzcd}
\qquad (Y')^\bullet\arrow[d,"f^\bullet"] \ :  0\arrow[r] &P(6^*)\arrow[d,"f^{-2}\,=\,id"]\arrow[r,"\partial^{-2}"] 
                   & \Omega F(45')\arrow[d,"f^{-1}\,=\,{\mathrm{inclusion}}"]\arrow[r,"\partial^{-1}"] 
                   &0\arrow[r]\arrow[d,"f^0\,=\,0"]&0
                   \\
\quad\ (X')^\bullet :
0 \arrow[r]&\arrow[r,"d^{-2}"] P(6^*) &\arrow[r,"d^{-1}"] P(3^*)\oplus P(6)&0\arrow[r]&0
\end{tikzcd}
\]
that is a commutative diagram, namely $f^\bullet\in\Hom_{\mathcal C}((Y')^\bullet,(X')^\bullet)$
because of the definitions of $f^{-2}$ and $f^{-1}$. Hence we have
$\overline{f^\bullet}:={\mathrm{cls}}(f^\bullet)\in\Hom_{\mathcal K}((Y')^\bullet,(X')^\bullet)$.
Then, it is quite easy to know that $\overline{f^\bullet}$ is quasi-isomorphic
by making use of the shape of $\Omega F(45')$ in Lemma~\ref{mu'}(vii) 
and that of $P(6^*)$ in Lemma~\ref{P(6)-P(15)}(ii), and also Lemma~\ref{complex}.
This means that $\overline{f^\bullet}$ induces an isomorphism in $\mathcal D$.
\end{proof}

\begin{Lemma}\label{Ric02_6-1-(a)}
Condition (a) in \cite[Theorem 6.1]{Ric02} is satisfied.
\end{Lemma}
\begin{proof}
We claim first that $\Hom_{\mathcal D}(X^\bullet(45'),X^\bullet(45')[-1])=0$.
Set $Z^\bullet:=(Y')^\bullet[-1]$ and take any $f^\bullet\in\Hom_{\mathcal C}(Y^\bullet,Z^\bullet)$, namely
\[
\begin{tikzcd}
\qquad (Y')^\bullet\arrow[d,"f^\bullet"] :  0\arrow[r] &P(6^*)\arrow[d,"f^{-2}\,=\,0"]
         \arrow[r,"\partial^{-2}"] 
                   & \Omega F(45')\arrow[d,"f^{-1}"]\arrow[r,"\partial^{-1}"] 
                   &0\arrow[r]\arrow[d,"f^0\,=\,0"]&0
                   \\
\quad  \ \ \ \quad Z^\bullet \ :\ 
0 \arrow[r]&\arrow[r,"\delta^{-2}"]0 &\arrow[r,"\delta^{-1}"] P(6^*)&\Omega F(45')\arrow[r]&0
\end{tikzcd}
\]
where $\delta^{-1}:=\partial^{-2}$. Obviously, $f^{-1}\circ\partial^{-2}=0$, we have 
$0=f^{-1}(\Im\,\partial^{-2})=f^{-1}(P(6)/6)$ by Lemma~\ref{mu'}(ii). So,
we can consider that $f^{-2}: \Omega F(45')/(P(6^*)/6^*)\rightarrow P(6^*)$.
Then, since $\Omega F(45')/(P(6^*)/6^*)\cong 15^*$ by Lemmas~\ref{mu'}(viii) and 
\ref{P(6)-P(15)}(ii), and since $c_{6^*,15^*}=0$ by Lemma~\ref{P(6)-P(15)}(ii),
it must be that $f^{-1}=0$. This shows that $f^\bullet=0$ even in $\mathcal C$.
Thus, of course $\Hom_{\mathcal D}(X^\bullet(45'),X^\bullet(45')[-1])=0$.
So, the claim is proved. All the other cases are proved by similar way and
by much more simple way. 
\end{proof}

\begin{Lemma}\label{Ric02_6-1-(c)}
The set $\{X^\bullet(S)\,|\,S\in\Irr_k(A) \}$
generates $\mathcal D$ as a triangulated category,
where $\Irr_k(A):=\{ k_G,10,10',15,15^*,34,45,45'\}$.
Namely, Condition (c) in \cite[Theorem 6.1]{Ric02} is satisfied.
\end{Lemma}
\begin{proof}
Let $\mathcal X$ be the full subcategory of $\mathcal D$ generated by the set
$\{X^\bullet(S)\,|\,S\in\Irr_k(A)\}$. 
We claim that $\mathcal X=\mathcal D$.
For an object $Y$ in $\mathcal D$, we simply write $Y\in\mathcal X$ when $Y$ is
an object of $\mathcal X$.
By Lemmas~\ref{complex} and \ref{F(k)F(15)F(15*)}, 
\begin{equation}\label{k_H}
k_H,15,15^*, 7\in\mathcal X.
\end{equation}
Hence $L:=\Omega^{-1}(15)\in\mathcal X$.
Then, Lemma~\ref{P(6)-P(15)}(i) implies that
$L$ contains the simple $3$ as a composition factor
with multiplicity one, and all the simples other than
$3$ in the composition factors of $L$ are 
in $\mathcal X$ by (\ref{k_H}), and hence $3\in\mathcal X$.
Similar for $3^*$, which means that
\begin{equation}\label{3}
3, 3^*\in\mathcal X.
\end{equation}
Next, since $X^\bullet(45)\in\mathcal X$, the information on its
cohomology groups in Lemma~\ref{complex} tells us that $6\in\mathcal X$
by using (\ref{k_H}) and (\ref{3}).
Similar for $6^*$, so that
\begin{equation}\label{6}
6, 6^*\in\mathcal X.
\end{equation}
Since $\Irr_k(B)=\{k_H,3,3^*,6,6^*,7,15,15^*\}$ by Notation~\ref{Notation1},
it follows from (\ref{k_H}), (\ref{3}) and (\ref{6}) that all the simples in $B$
are in $\mathcal X$, which surely implies that $\mathcal X=\mathcal D$.
\end{proof}

\section{Proofs of the Results}
\begin{proof}[Proof of Theorem~\ref{MainTheorem}]
Follows by Lemmas~\ref{complex}$\sim$\ref{Ric02_6-1-(c)}
except Lemma~\ref{Hom(X(45),X(45))},
\cite[Theorem~6.1]{Ric02} and \cite[Theorem~4.4]{Hol03}.
\end{proof}

\begin{proof}[Proof of Corollary~\ref{Cor}]
Set $\tilde G:=\Aut(G)$ and $\tilde H:=\Aut(H)$, so that
$\tilde G=G\rtimes\langle\sigma\rangle\cong G\rtimes C_2$
and $\tilde H=G\rtimes\langle\tau\rangle\cong H\rtimes C_2$
by \cite[p.31 and p.13]{Atlas}.
Since $|\Out({\sf M}_{12})|=|\Out(\SL_3(3))|=2$ that is prime to our $p:=3$,
Clifford's/Maschke's theorems and essentially the same method taken in the proof of 
Theorem~\ref{MainTheorem}
do work to prove this corollary.
Precisely speaking \cite[Theorem 2]{Kos87} and \cite[Theorem 2]{KW99a} imply the result
(see the following table).

\begin{table}[h]
\vspace{-1cm}
\caption*{}
%\label{table2}
\centering
 \begin{tabular}{l |ccccccccc| l}
{\rm degree} &$-3$& &$-2$ && $-1$ && $0$ &&$1$& 
   ${\mathrm{H}}^i:={\mathrm{H}}^i(X^\bullet(\widetilde S))$ \\
 \hline
&&&&&&&&& \\
$X^\bullet(k_{\tilde G})$ & $0$&$\rightarrow$& $0$ &$\rightarrow$& $0$ 
     &$\rightarrow$& $\tilde F(k_{\tilde G})$ &$\rightarrow$&$0$& $\H^0=k_{\tilde H}$ \\ 
$X^\bullet(\tilde 1)$ & $0$&$\rightarrow$& $0$ &$\rightarrow$& $0$ 
 &$\rightarrow$& $\tilde F(\tilde 1)$ &$\rightarrow$&$0$& ${\mathrm{H}}^0=\tilde 1$ \\     
$X^\bullet(\tilde{34})$ & $0$&$\rightarrow$&$0$ &$\rightarrow$& $0$ 
   &$\rightarrow$& $\tilde F(\tilde{34})$ &$\rightarrow$&$0$& ${\mathrm{H}}^0={\tilde 7}$ \\ 
 $X^\bullet(\tilde{34'})$ & $0$&$\rightarrow$&$0$ &$\rightarrow$& $0$ 
   &$\rightarrow$& $\tilde F(\tilde{34'})$ &$\rightarrow$&$0$& ${\mathrm{H}}^0={\tilde 7'}$ \\       
$X^\bullet(\tilde{30})$ & $0$&$\rightarrow$&$0$ &$\rightarrow$& $0$ &$\rightarrow$
   & $\tilde F(\tilde{30})$ &$\rightarrow$&$0$& ${\mathrm{H}}^0=\tilde{30}$ \\   
%     
%$X^\bullet(15^*)$ & $0$&$\rightarrow$&$0$ &$\rightarrow$& $0$ &$\rightarrow$& $F(15^*)$ &
%     $\rightarrow$&$0$& ${\mathrm{H}}^0=15^*$ \\     
%
$X^\bullet(\tilde{20})$ & $0$&$\rightarrow$&$0$ &$\rightarrow$& $P(\tilde{12})$ 
     &$\rightarrow$& $\tilde F(\tilde{20})$ &$\rightarrow$&$0$
     & ${\mathrm{H}}^{-1}=\tilde V$    \\  
     %
%$X^\bullet(10')$ & $0$&$\rightarrow$&$0$ &$\rightarrow$& $P(6^*)$ &$\rightarrow$& $F(10')$ &
%     $\rightarrow$&$0$& ${\mathrm{H}}^{-1}=U(k_H,7,3^*,6^*)$ \\   
%
$X^\bullet(\tilde{90})$ & $0$&$\rightarrow$&$P(\tilde{12})$ &$\rightarrow$
 &$P(\tilde 6)\oplus P(\tilde{12})$ &$\rightarrow$& $\tilde F(\tilde{90})$ &
     $\rightarrow$&$0$& ${\mathrm{H}}^{-2}=\tilde{12}, \ \  {\mathrm{H}}^{-1}=\tilde{30}$ 
\end{tabular}
\end{table}

\noindent
where 
\[
\begin{tikzpicture}
\draw
(-1,-1)node[left](){$\tilde V:=$}
(-0.5,-0.3)node[below](){\ \ $k_{\tilde H}$}
(0.5,-0.8)node[above](){$\tilde 1$}
(-0.5,-0.7)--(-0.5,-1)node[below](){$\tilde 7$}
(0.5,-0.7)--(0.5,-1)node[below](){$\tilde{7'}$}
(-0.5,-1.5)--(0,-1.8)node[below](){$\tilde 6$}
(0.5,-1.5)--(0.1,-1.8)
(0,-2.3)--(0,-2.6)node[below](){$\tilde{12}$}
;
\end{tikzpicture}
.
\]
\end{proof}

\begin{proof}[Proof of Corollary~\ref{HH^1}]
Follows from Theorem~\ref{MainTheorem} and \cite[Proposition 2.21.9]{Lin18}.
%\cite[Proposition 2.5]{Ric91}.
\end{proof}

\begin{Remark}
Compare Theorem~\ref{MainTheorem} with
\cite[Proposition 6.2 (S3), (S5)]{AE25}.
\end{Remark}

%\pagebreak
\noindent {\bf Acknowledgment.} 
{\small
The first author is grateful to Alexander Zimmermann for answering questions
on his book \cite{Zim14} and also to Markus Linckelmann for useful information.
    }

\end{document}